\documentclass{amsart}
\usepackage{amsmath}
\usepackage{amssymb}
\usepackage{fancybox}
\usepackage{eucal}
\usepackage{amsthm}
\usepackage{amscd}
\usepackage{hyperref}
\usepackage{wasysym}
\usepackage{url}
\usepackage{bbm}
\usepackage{mathrsfs}
\usepackage{amsfonts}
\usepackage{appendix}

\usepackage{amssymb}
\usepackage[]{amsmath, amsthm, amsfonts,graphicx, amscd,}
\usepackage[all,cmtip]{xy}
\input amssym.def 
\input amssym
\DeclareMathOperator*{\forkindep}{\raise0.2ex\hbox{\ooalign{\hidewidth$\vert$\hidewidth\cr\raise-0.9ex\hbox{$\smile$}}}}

\makeatletter
\let\@wraptoccontribs\wraptoccontribs
\makeatother

\DeclareMathOperator{\ord}{ord}

\begin{document}

\newtheorem{thm}{Theorem}[section]
\newtheorem{theorem}[thm]{Theorem}
\newtheorem{lemma}[thm]{Lemma}
\newtheorem{claim}[thm]{Claim}
\newtheorem{definition}[thm]{Definition}
\newtheorem{corollary}[thm]{Corollary}
\newtheorem{aside}[thm]{Aside}
\newtheorem{corr}[thm]{Corollary}
\newtheorem{fact}[thm]{Fact}
\newtheorem {cl}[thm]{Claim}
\newtheorem*{thmstar}{Theorem}
\newtheorem{prop}[thm]{Proposition}
\newtheorem*{propstar}{Proposition}
\newtheorem {lem}[thm]{Lemma}
\newtheorem*{lemstar}{Lemma}
\newtheorem{conj}[thm]{Conjecture}
\newtheorem{question}[thm]{Question}
\newtheorem*{questar}{Question}
\newtheorem{ques}[thm]{Question}
\newtheorem*{conjstar}{Conjecture}
\theoremstyle{remark}
\newtheorem{rem}[thm]{Remark}
\newtheorem{np*}{Non-Proof}
\newtheorem*{remstar}{Remark}
\theoremstyle{definition}
\newtheorem{defn}[thm]{Definition}
\newtheorem*{defnstar}{Definition}
\newtheorem{exam}[thm]{Example}
\newtheorem{example}[thm]{Example}
\newtheorem*{examstar}{Example}
\newcommand{\pd}[2]{\frac{\partial #1}{\partial #2}}
\newcommand{\pp}{\partial }
\newcommand{\pdtwo}[2]{\frac{\partial^2 #1}{\partial #2^2}}
\def\Ind{\setbox0=\hbox{$x$}\kern\wd0\hbox to 0pt{\hss$\mid$\hss} \lower.9\ht0\hbox to 0pt{\hss$\smile$\hss}\kern\wd0}
\def\Notind{\setbox0=\hbox{$x$}\kern\wd0\hbox to 0pt{\mathchardef \nn=12854\hss$\nn$\kern1.4\wd0\hss}\hbox to 0pt{\hss$\mid$\hss}\lower.9\ht0 \hbox to 0pt{\hss$\smile$\hss}\kern\wd0}
\def\ind{\mathop{\mathpalette\Ind{}}}
\def\nind{\mathop{\mathpalette\Notind{}}} 
\newcommand{\m}{\mathbb }
\newcommand{\mc}{\mathcal }
\newcommand{\mf}{\mathfrak }
\newcommand{\is}{^{p^ {-\infty}}}
\newcommand{\codim}{\operatorname{codim}}
\newcommand{\Gal}{\operatorname{Gal}}
\newcommand{\Num}{\operatorname{Num}}
\newcommand{\Cl}{\operatorname{Cl}}
\newcommand{\Div}{\operatorname{Div}}
\newcommand{\Ann}{\operatorname{Ann}}
\newcommand{\Frac}{\operatorname{Frac}}
\newcommand{\lcm}{\operatorname{lcm}}
\newcommand{\height}{\operatorname{ht}}
\newcommand{\Der}{\operatorname{Der}}
\newcommand{\Pic}{\operatorname{Pic}}
\newcommand{\Sym}{\operatorname{Sym}}
\newcommand{\Proj}{\operatorname{Proj}}
\newcommand{\characteristic}{\operatorname{char}}
\newcommand{\Spec}{\operatorname{Spec}}
\newcommand{\Hom}{\operatorname{Hom}}
\newcommand{\res}{\operatorname{res}}
\newcommand{\Aut}{\operatorname{Aut}}
\newcommand{\length}{\operatorname{length}}
\newcommand{\Log}{\operatorname{Log}}
\newcommand{\Set}{\operatorname{Set}}
\newcommand{\Fun}{\operatorname{Fun}}
\newcommand{\id}{\operatorname{id}}
\newcommand{\Gp}{\operatorname{Gp}}
\newcommand{\Ring}{\operatorname{Ring}}
\newcommand{\Mod}{\operatorname{Mod}}
\newcommand{\Mor}{\operatorname{Mor}}
\newcommand{\SheafHom}{\mathcal{H}om}
\newcommand{\pre}{\operatorname{pre}}
\newcommand{\coker}{\operatorname{coker}}
\newcommand{\acl}{\operatorname{acl}}
\newcommand{\dcl}{\operatorname{dcl}}
\newcommand{\tp}{\operatorname{tp}}
\newcommand{\dom}{\operatorname{dom}}
\newcommand{\val}{\operatorname{val}}
\newcommand{\Aa}{\mathbb{A}}
\newcommand{\Qq}{\mathbb{Q}}
\newcommand{\Rr}{\mathbb{R}}
\newcommand{\Zz}{\mathbb{Z}}
\newcommand{\Nn}{\mathbb{N}}
\newcommand{\Cc}{\mathbb{C}}
\newcommand{\Ii}{\mathbb{I}}
\newcommand{\Gg}{\mathbb{G}}
\newcommand{\Uu}{\mathbb{U}}
\newcommand{\Mm}{\mathbb{M}}
\newcommand{\Ff}{\mathbb{F}}
\newcommand{\Pp}{\mathbb{P}}

\def\pr{\operatorname{pr}}
\def\H{\operatorname{H}}
\def\bu{{\mathbf{u}}}
\def\bv{{\mathbf{v}}}
\def\bx{{\mathbf{x}}}
\def\V{\mathbb{V}}
\def\P{\mathbb{P}}
\def\L{\mathbb{L}}
\def\Y{\mathbb{Y}}
\def\I{\mathbb{I}}
\def\den{{\operatorname{den}}}
\def\num{{\operatorname{num}}}
\def\card{{\operatorname{card}}}
\def\CI{{\mathcal{I}}}
\def\ld{{\operatorname{ld}}}
\def\init{{\operatorname{init}}}
\def\Sep{\operatorname{S}}
\def\chow{{\operatorname{Chow}}}
\def\denom{{\operatorname{denom}}}
\def\sat{{\operatorname{sat}}}
\def\asat{{\operatorname{asat}}}
\def\ff{\mathcal{F}}
\def\VB{{\mathbf{V}}}
\def\id{{\operatorname{id}}}
\def\lv{{\operatorname{lv}}}
\def\trdeg{\operatorname{tr.deg}}
\def\dtrdeg{\operatorname{d.tr.deg}}

\title[Differential Chow varieties]{Differential Chow varieties exist} 
\author{James Freitag}
\email{freitag@math.berkeley.edu}
\address{University of California, Los Angeles \\
Department of Mathematics \\
520 Portola Plaza \\
Math Sciences Building 6363 \\
Los Angeles, CA 90095-1555 \\
USA}

\author{Wei Li}
\email{liwei@mmrc.iss.ac.cn}
\address{Key Lab of Mathematics Mechanization \\
Academy of Mathematics and Systems Science \\
Chinese Academy of Sciences \\
No. 55 Zhongguancun East Road, 100190, Beijing \\
China}

\author{Thomas Scanlon}
\email{scanlon@math.berkeley.edu}
\address{University of California, Berkeley \\
Department of Mathematics \\
Evans Hall \\
Berkeley, CA 94720-3840 \\
USA}

\contrib[with an appendix by]{William Johnson}

\thanks{JF is partially supported by an NSF MSPRF. WL is partially supported by NSFC Grant 11301519 and thanks 
the University of California, Berkeley for providing a good research environment during her appointment as 
a Visiting Scholar. TS is partially supported by NSF Grant DMS-1363372.}

\begin{abstract} The Chow variety is a parameter space for effective algebraic cycles on $\mathbb P^n$ (or $\mathbb{A}^n$) of given dimension and degree. 
We construct its analog for differential algebraic cycles on $\mathbb{A}^n$, answering a question of~\cite{GDIT}. 
The proof uses the construction of classical algebro-geometric Chow varieties, the theory of characteristic sets of 
differential varieties and algebraic varieties, the theory of prolongation spaces, and the theory of differential Chow forms. In the course of the proof several 
definability results from the theory of algebraically closed fields are required.
Elementary proofs of these results are given in the appendix. 

\end{abstract}
\maketitle

\section{Introduction}

For simplicity in the following discussion, let $k$ be an algebraically closed field
and $\mathbb{P}^n$ the projective space over $k$. 
The $r$-cycles on $\mathbb{P}^n$  are elements 
of the free $\m Z$-module generated by irreducible varieties in $\mathbb{P}^n$  of dimension $r$. 
If the coefficients are taken over $\m N$ then the cycle is said to be positive or effective. 
For a given effective cycle $\sum n_i V_i $, the degree of $\sum n_i V_i$ is given by $\sum n_i 
\cdot \deg (V_i)$ where the degree of the variety is computed with respect to 
some fixed very ample line bundle on $V$.  
The positive $r$-cycles of degree $d$ on $\mathbb{P}^n$ are parameterized by a $k$-variety, called the Chow variety. 
For background on Chow varieties and Chow forms, see~\cite{Chow}
(or~\cite{harris1992algebraic} or~\cite{dalbec1995introduction} for a modern exposition). The purpose of this article is to carry out the 
construction of the differential algebraic analog of the Chow variety, whose construction was begun in~\cite{GDIT}, but was completed only in 
certain very special cases. For our purposes, one can view Chow varieties and their differential counterparts as parameter spaces for cycles 
with particular characteristics (degree and codimension in the algebraic case). The algebraic theory of Chow varieties also has numerous 
applications and deeper uses (e.g. Lawson (co)homology~\cite{friedlander1989homology} and various counting problems in 
geometry~\cite{eisenbud20103264}). 

Working over a differentially closed field $K$ with a single derivation and 
natural number $n$, the 
group of differential cycles of dimension $d$ and order $h$ in affine $n$-space over
$K$ is the free $\m Z$-module generated by irreducible differential 
subvarieties $W  \subseteq  {\mathbb A}^n$ so that the dimension, $\dim (W)$, is $d$ and the order, $\ord (W)$, is $h$. 
The differential cycles of index $(d, h, g, m)$ are those 
effective cycles with leading differential degree $g$ and differential degree $m$. These invariants have a very natural definition and are a suitable notion 
of degree for differential cycles; see section \ref{Chowsection} for the definitions. Our main result establishes the existence of a 
differential variety which parameterizes this particular set of effective differential cycles of ${\mathbb A}^n$. 

There are various foundational approaches to differential algebraic geometry (e.g. the scheme-theoretic approaches~\cite{KovacicKolchin} or the 
Weil-style definition of an abstract differential algebraic variety~\cite{KolchinDAG}). Such abstract settings will not be pertinent here, since 
we work exclusively with differential algebraic subvarieties of affine spaces over a differential field. In this setting, beyond the basic development of the 
theory, there are two approaches relevant to our work. The first is the classical theory using characteristic sets~\cite{KolchinDAAG}. We also use 
the more recent geometric approach using the theory of jet and prolongation spaces~\cite{MSJETS}. This approach allows one to replace a 
differential algebraic variety by an associated sequence of algebraic varieties, but owing to the Noetherianity of the Kolchin topology, some 
finite portion of the sequence contains all of the data of the sequence. This allows the importation of various results and techniques from the 
algebraic category. 

We use the two approaches in the following manner. We use classical algebraic Chow varieties to parameterize prolongation sequences. We then use 
the dominant components of these prolongation sequences to parameterize the \emph{characteristic sets} of differential algebraic cycles with 
given index. There are essentially two steps to the construction. First,  degree bounds are used to restrict the space of Chow 
varieties in which we must look for the points which generate the prolongation sequences parameterizing differential cycles of a given index; 
this development uses basic intersection theory (e.g.~\cite{Heintz}) and the theory of differential Chow forms~\cite{GDIT}. Within the appropriate 
Chow varieties which parameterize these prolongation spaces, only a subset of the points will correspond to differential cycles with the 
specified index. We show that this collection of points (such that the the dominant component  of the differential variety corresponding to the 
prolongation sequence generated by each irreducible component of the algebraic cycle represented by this point has the specified index, to be 
defined in Definition \ref{def-index}) is in fact a differentially constructible subset.

To say more precisely what we mean by the existence of a differential Chow variety we should speak about 
representable functors.   For a fixed ambient dimension $n$ and 
index $(d, h, g, m)$ we may associate to each differential field $k$ the set of differential
cycles of index $(d, h, g, m)$ in affine $n$-space over $k$ thereby obtaining a functor
$\Cc(n, d, h, g, m): \delta\text{-}\operatorname{Field} \to \operatorname{Set}$ from 
the category of differential fields to the category of sets.  (On morphisms of 
differential fields, this functor is given by base change.)  Of course, each differentially constructible set 
$X$ (defined over $\mathbb Q$) gives the functor of points  
$h_X: \delta\text{-}\operatorname{Field} \to \operatorname{Set}$ determined by
$k \mapsto X(k)$.    We shall show that $\Cc(n, d, h, g, m)$
is representable, meaning that there is a differentially constructible set which we shall 
call $\delta\text{-}\chow(n, d, h, g, m)$ and a natural isomorphism
between the functors $\Cc(n, d, h, g, m)$ and 
$h_{\delta\text{-}\chow(n, d, h, g, m)}$.  In fact, we prove 
a little more in that we produce a universal family of cycles over $\delta\text{-}\chow(n, d, h, g, m)$
so that the natural transformation from $\delta\operatorname{-Chow}(n, d, h, g, m)$ to 
$\Cc(n, d, h, g, m)$ is given by taking fibers of this family.   

The construction of differential Chow varieties is related to \emph{canonical 
parameters} in the sense of model theory. In the theory of 
differentially closed fields, canonical parameters manifest themselves
as the generators of fields of definition of differential varieties. In 
recent 
years, detailed analyses of canonical parameters have been undertaken in 
analogy with results of Campana~\cite{campana1980algebricite} and 
Fujiki~\cite{fujiki1982douady} from compact complex manifolds (for instance, 
see~\cite{PillayZiegler, chatzidakis2012note, moosa2008canonical}). 
The following is essentially pointed out by Pillay and Ziegler~\cite{PillayZiegler}. Let 
$K$ be a differential field and $ x$ an $n$-tuple of elements from 
some differential field extension. Let $X$ be the differential locus of 
$x$ over $K$.  
Let $L$ be a differential field extension of $K$ and let $Z$ be the differential locus of $x$ over 
$L$.  Let $b$ be a generator for the 
differential field of definition of $Z$ over $K$.  That is, 
there is some differential algebraic subvariety $Y  \subseteq   {\mathbb A}^n \times {\mathbb A}^m$
so that $b \in {\mathbb A}^m(L)$, $(x,b) \in Y_b$ and the second projection map 
$\pi:Y \to {\mathbb A}^m$ is differentially birational over its image.   
Consider $\mc Y := {_{ x}}Y$, the fiber of $Y$ over $ x$ via the first projection 
$Y \to {\mathbb A}^n$, which is subvariety of a certain differential Chow variety of $X$ in the
sense of this paper.   The main result of~\cite{PillayZiegler} is that $\mc Y$ is internal to the 
constants.  One could certainly expand upon these observations to give statements about the 
structure of differential Chow varieties, but we will not pursue these matters in further detail in this paper. 

In~\cite[page 581]{PillayZiegler}, Pillay and Ziegler write of the above situation, 
\begin{quote} We are unaware of any systematic development of machinery and language (such as ``differential Hilbert spaces") in differential algebraic geometry which
is adequate for the geometric translation above. This is among the reasons
why we will stick with the language of model theory in our proofs below.
The issue of algebraizing the content and proofs is a serious one which will
be considered in future papers.
\end{quote}
\noindent Subsequent work by Moosa and Scanlon~\cite{MSJETS} did algebraize and generalize much of the work by Pillay and Ziegler, but no systematic development of differential Hilbert schemes or differential Chow varieties appears to have occurred in the decade following Pillay and Ziegler's work. One should view~\cite{GDIT} as the beginning of such a systematic development, where the theory of the differential Chow form was developed, and the existence of differential Chow varieties was established in certain very special cases. In~\cite[section 5]{GDIT}, the authors write that they are unable to prove the existence of the differential Chow variety in general. The work here is an extension of~\cite{GDIT}, in which we will establish the existence of the differential Chow variety in general, answering the most natural question left open by~\cite{GDIT}. As we have pointed out above, our general technique is also the descendant of a line thinking that originated (at least in the model theoretic context) with Pillay and Ziegler's work on jet spaces and the linearization of differential equations. 

While the theory of canonical parameters provides an ad hoc solution to the problem of 
parametrizing the differential subvarieties of a given affine space, the theory of 
differential Chow varieties is superior in some respects.   Firstly, the theory of 
differential Chow varieties provides a natural stratification of the parameter spaces 
via the discrete index invariant.  Secondly, the conceptual definition of the differential Chow 
varieties through the notion of a differential Chow form permits one to 
effectively compute the differential Chow coordinates of a differential variety.  Indeed, 
this is done in detail for prime differential ideals in~\cite{li2014}.   
To compute a canonical parameter using elimination of
imaginaries relative to the theory of differentially closed fields of characteristic zero 
would require an appeal to algebraic invariant theory which itself could be made only after 
computing nontrivial bounds on the order of the generators of the eventual quotient. Even after this process is completed, the family of varieties one 
obtains is not characterized by differential algebraic invariants; this is not the case with the differential Chow varieties of this work, 
which \emph{are} characterized by specifying natural differential algebraic invariants.  On the 
other hand, there is a cost to working with differential Chow coordinates: one must prove that the 
Chow coordinates of differential varieties with a given index actually form a differentially constructible
set; this is precisely what we do in this paper.

%While the theory of canonical parameters provides an ad hoc solution to the problem of
%parametrizing the differential subvarieties of a given affine space, it does not on its 
%own provide an effective computational method to produce such codes.   
%Methods to explicitly produce the differential Chow coordinates of differential 
%varieties are given in~\cite{GDIT} based on the theory of characteristic sets.   

The rest of the paper is organized as follows.
In section \ref{background}, we give background definitions and some preliminary results which we use later in the paper. In addition, we describe the 
relationship of the problems we consider to the Ritt problem. Following this interlude, we prove the results which eventually allow us to work around the 
issues involved in the Ritt problem (whose solution would allow for a simplification of the proofs of the results in this paper). 
In section \ref{Chowsection}, we describe the necessary background from the classical theory of Chow varieties and introduce the theory of differential Chow forms. 
%Our approach is slightly nonstandard in this section, owing to the fact that we work with affine varieties. 
In section \ref{confine}, we establish various bounds on the order and degree of the 
varieties we consider using the theory of differential Chow forms. In section \ref{deltaChow-1}, we establish the existence of differential Chow 
varieties, proving the main result of the paper.

The appendix gives elementary proofs of several facts from algebraic geometry which we
 require. The facts proved in the appendix are well known and are frequently used in model
 theory (for instance, see the citation in appendix 3.1 of~\cite{HrIt}), however, several of the proofs in the appendix seem to be new. Constructive proofs of the results in the appendix (which give 
 additional information about certain bounds, rather than simply proving that a certain 
 bound exists) are much more involved (see~\cite{seidenberg1974constructions}, which 
 corrected and modernized some of the proofs given in~\cite{hermann1926frage}). Various other non-constructive 
 proofs of the theorem are given in the literature~\cite[15.5.3]{EGA4}~\cite[where a 
 nonstandard approach is taken]{vandenDriesbounds}~\cite[where a model theoretic 
 approach is taken to give an elementary proof]{hrushovski1992strongly}.

\section{Preliminaries and Prolongations} \label{background}
We fix $\mc U$ a saturated differentially closed field with a single derivation, $\delta$.   Implicitly, 
all differential fields we consider are subfields of $\mc U$. If $V$ is 
an algebraic variety or a differential variety, then an expression of the
form ``$a \in V$'' is shorthand for ``$a \in V(\mc U)$''.   
Throughout, $K$ will be a small differential subfield of $\mc U$ and
$\delta$ denotes the distinguished derivation on $K$, $\mc U$, or, indeed,
any differential ring that we consider.  Unless explicitly stated to the 
contrary, all varieties and differential varieties are defined over $K$  and have coordinates in $\mathcal{U}$. 
By convention, $\mathbb{A}^n$ and $\mathbb{A}^n(\mathcal{U})$ stand for the affine space with coordinates in $\mathcal{U}$.

If $f:X \to Y$ is a morphism of varieties,
then by $f(X)$ we mean the scheme theoretic image of $X$ under $f$.  That is, 
$f(X)$ is the smallest subvariety $Z$  of $Y$ for which $f$ factors through the inclusion 
$Z \hookrightarrow Y$.  On points, $f(X)$ is the Zariski closure of 
$\{ f(a) ~:~ a \in X(\mc U) \}$.  

We write $K \{ x_1, \ldots, x_n \}$ for the differential polynomial ring 
in the variables $x_1, \ldots, x_n$ over $K$.   For $m \in \m N$, we write 
$$K \{ x_1, \ldots, x_n \}_{\leq m} = 
K [x_i^{(j)} ~:~ 1 \leq i \leq n, 0 \leq j \leq m]$$ for the 
subring of 
differential polynomials of order at most $m$ where we have written 
$x_i^{(j)}$ for $\delta^j x_i$. We also use $x_i^{[m]}$ to denote the set $\{x_i^{(j)}:j=0,1,\ldots,m\}$
and sometimes denote $K \{ x_1, \ldots, x_n \}_{\leq m}=K[x_1^{[m]},\ldots,x_n^{[m]}]$. 

If $\mathcal{I}  \subseteq   K \{ x_1, \ldots, x_n \}$ is a differential ideal, then we write $\V(\mathcal{I})$ for
the differential subvariety of $\m A^n(\mc U)$ defined by the vanishing of all 
$f \in \mathcal{I}$, and for $S  \subseteq   \m A^n (\mc U)$, we 
let $\I(S)  \subseteq   K \{ x_1, \ldots, x_n \}$ be the differential ideal
of all differential polynomials over $K$ vanishing on $S$.  On the other hand, if $I  \subseteq   K [ x_1, \ldots , x_n ]$ is an ideal, then we write $V(I)$ for the variety defined by the vanishing of all $f \in I$, and for $S  \subseteq   \m A^n (\mc U), $ we write $I(S)$ for the ideal of polynomials in $K[x_1, \ldots , x_n]$ which vanish on $S$.

In general, if 
$R$ is a commutative ring, we write $\mc Q (R)$ for its total ring of fractions.
When $R$ is a differential ring, so is $\mc Q (R)$.  For a differential variety $V$, we write 
$K \langle V \rangle$ for 
$\mc Q ( K \{ x_1, \ldots, x_n \} / \I(V) )$.   When $\I(V)$ is prime, that is, when 
$V$ is irreducible, this is called the differential function field of $V$.
For $S  \subseteq   
K \{ x_1, \ldots, x_n \}$ we write $(S)$ for the ideal generated 
by $S$ and $[S]$ for the \emph{differential} ideal generated by $S$. 
When $S = \{ f \}$ is a singleton, we write $(f) := (S)$ and $[f] := [S]$.  
Likewise, we write $ \V(f)$ for $\V( [f] )$.  

We sometimes speak about ``generic points''.  These should be understood in the 
sense of Weil-style algebraic (or differential algebraic) geometry.  That is,
if $V$ is a variety (respectively, differential algebraic variety) over $K$, then 
$\eta \in V(\mc U)$ is generic if there is no proper subvariety (respectively,
differential subvariety) $W \subsetneq V$ defined over $K$ with 
$\eta \in W(\mc U)$.  Provided that $V$ is irreducible, this is equivalent 
to asking that the field $K(\eta)$ (respectively, differential field
$K \langle \eta \rangle$) be isomorphic over $K$ to $K(V)$ (respectively, 
$K \langle V \rangle$).  We will also say that a point is a generic point of some 
ideal (or differential ideal) if it is a generic point of the corresponding variety (or differential
variety). 
  
 \subsection{Methods of algebraic and differential characteristic sets }
In this paper, the Wu-Ritt characteristic set method is a basic tool
for establishing a correspondence between differential algebraic cycles 
and algebraic cycles satisfying certain conditions. In
this section we recall the definition and basic properties of 
algebraic  and  differential characteristic sets~\cite{Wu, Ritt}.

First, we introduce the algebraic characteristic set method. 
Consider the polynomial ring $K[x_1,\ldots,x_n]$ and fix an ordering on 
$x_1,\ldots,x_n$, say, $x_1<\cdots < x_n$. 
Given $f\in K[x_1,\ldots,x_n]\backslash K$, the {\em leading variable} of $f$ is the greatest variable $x_k$ 
effectively appearing in $f$, denoted by $\lv(f)$. 
Regarding $f$ as a univariate polynomial in $\lv(f)$, the leading coefficient of $f$ is called the {\em initial} of $f$, denoted by $\init(f)$. 
A sequence of polynomials $\langle A_1, \ldots, A_r\rangle$ is said to be an {\em ascending 
chain}, if   either
\begin{enumerate}
\item	$r = 1$ and $A_1\neq 0$,  or
\item  all the $A_i$ are nonconstant,  $\lv(A_i) <\lv(A_j)$ for $1\leq i < j\leq r$ and $\deg(A_k, \lv(A_k)) > \deg(A_m, \lv(A_k))$ for $m > k$. 
(Here $\deg(f,y)$ denotes the degree of $f$ regarded as a polynomial in the variable $y$.)
\end{enumerate}

Given two polynomials $f$ and $g$,  $f$ is said to be of higher rank than $g$ and denote $f>g$,
 if either $\lv(f)>\lv(g)$, or $x=\lv(f)=\lv(g)$ and $\deg(f,x)>\deg(g,x)$.  
If $\lv(f)=\lv(g)=x$ and $\deg(f,x)=\deg(g,x)$, then we say $f$ and $g$ have the same rank.
Suppose $\mathcal{A}= \langle A_1, \ldots, A_r \rangle$
and $\mathcal{B}= \langle B_1, \ldots, B_s \rangle$ 
are two ascending chains in $K[x_1,\ldots,x_n]$. 
We say $\mathcal{A}$ is of lower rank than $\mathcal{B}$, 
denoted by $\mathcal{A}\prec\mathcal{B}$,  if either 
\begin{enumerate}
\item there exists $k \leq \min\{r,s\}$ such that for $i<k$, $A_i$ and $B_i$ have the same rank, and 
 $A_k<B_k$, or
\item $r>s$, and for each $i\leq s$, $A_i$ and $B_i$  have the same rank.
\end{enumerate}
 
Given an ideal $\mathcal{I}$ in $K[x_1,\ldots,x_n]$, 
an ascending chain contained in $\mathcal{I}$
which is of lowest rank is called an algebraic  characteristic set of 
$\mathcal{I}$. If $V \subseteq   \m A^n$ is an irreducible variety, then an algebraic  characteristic set of 
$V$ is defined as a characteristic set of its corresponding prime ideal $I(V)$.  

Let $\mathcal {A}= \langle A_{1} ,A_{2}, \ldots, A_{t} \rangle$ be an ascending chain.
We call $\mathcal{A}$  an {\em irreducible  ascending chain} if
for any $1\leq i\leq t$, there can not exist any relation of the
form $$T_iA_i=B_iC_i,\mod \,(A_1,\ldots,A_{i-1})$$
where $B_i,C_i$ are polynomials with the same leader as $A_i$, $T_i$
is a polynomial with lower leader than $A_i$, and $B_i, C_i, T_i$
are reduced with respect to $A_1,\ldots,A_{i-1}$ (\cite{Wu}). In other
words, an ascending chain $\mathcal {A}$ is irreducible
if and only if  there exist no polynomials $P$ and $Q$ which are
reduced with respect to $\mathcal A$ and $PQ\in\asat(\mathcal A)=(\mathcal
{A}):I_\mathcal{A}^\infty$, where $I_\mathcal{A}^\infty$ stands for
the set of all products of powers of $\init({A_i})$.
 By \cite[p.89]{Ritt},  for  an ascending chain $\mathcal{A}$ to be a characteristic set of a prime polynomial ideal,
it is  necessary and sufficient  that $\mathcal{A}$ is irreducible.

We now return to ordinary differential polynomial algebra and introduce differential characteristic methods for differential polynomials.
Fix a sequence of differential variables $x_1, x_2, \ldots, x_n$ and consider the differential polynomial ring $K\{x_1,\ldots,x_n\}$.
A differential ranking  is a total order  $\prec$ on 
the set $\Theta := \{ x_i^{(j)}: \,{i\leq n, j \in \m N}\}$
satisfying
\begin{itemize}
\item For all $\theta \in \Theta$,  $\delta \theta \succ\theta$ and
\item If $\theta_1 \succ\theta_2$, then  $\delta\theta_1 \succ \delta\theta_2$.
\end{itemize}

An orderly ranking is a differential ranking which satisfies in 
addition
\begin{itemize}
\item If $k > \ell$, then  $\delta^k x_i \succ \delta^\ell x_j$ for all $i$ and $j$.
\end{itemize}

Throughout the paper, we fix some orderly ranking $\mathscr R$, and when we talk about characteristic set methods  in the polynomial ring $K\{x_1,\ldots,x_n\}_{\leq \ell}$, the ordering on $(x_{i}^{(j)})_{1\leq i\leq n;j\leq \ell}$ induced by $\mathscr R$ is fixed.

Let $f$ be a differential polynomial in $K \{x_1, \ldots, x_n \}$.
The \emph{leader} of $f$, denoted by $\ld(f)$,  is the greatest  $v\in\Theta$ with respect to $\prec$ which appears effectively in $f$.
Regarding $f$ as a univariate polynomial in $\ld(f)$, its leading coefficient is called the \emph{initial} of $f$, denoted by $\init(f)$,
and the partial derivative of $f$ with respect to $\ld(f)$ is called the  \emph{separant} of $f$, denoted by $\Sep_{f}$.
For any two differential polynomials $f$, $g$ in $K \{x_1, \ldots, x_n\}$, $f$ is said to be of  lower rank than $g$, denoted by $f < g$, if 
\begin{itemize}
	\item  $\ld(f) \prec \ld(g)$ or 
	\item  $\ld(f)=\ld(g)$ and $\deg(f,\ld(f))<\deg(g,\ld(f))$.
\end{itemize}

The differential polynomial $f$ is said to be reduced with respect to $g$ if no proper derivative of $\ld(g)$ appears in $f$ and $\deg(f,\ld(g)) < \deg(g,\ld(g))$. 

Let $\mathcal {A}$ be a set of differential polynomials. Then $\mathcal {A}$ is said to 
be an auto-reduced set  if each differential polynomial in $\mathcal {A}$ is reduced 
with respect to any other element of $\mathcal {A}$. Every auto-reduced set is 
finite~\cite{Ritt}.

Let $\mathcal {A}$ be an auto-reduced set. 
We denote by $\H_{\mathcal {A}}$ the set of all initials and 
separants of $\mathcal {A}$ and by  $\H_\mathcal {A}^\infty$  the minimal multiplicative set 
containing $\H_\mathcal {A}$. The  saturation differential ideal  of $\mathcal {A}$ is 
defined to be
$$\sat(\mathcal {A})=[\mathcal {A}]:\H_{\mathcal {A}}^\infty=
\{f \in K\{x_1, \ldots, x_n  \} ~:~ \exists h\in \H_{\mathcal {A}}^\infty  \text{ for 
which } hf \in [\mathcal {A}] \} \text{ .}$$

An auto-reduced set $\mathcal {C}$ contained in a differential polynomial 
set $\mathcal {S}$ is said to be a \emph{characteristic set} of $\mathcal {S}$ 
if $\mathcal {S}$ does not contain any nonzero element reduced with respect to
 $\mathcal {C}$. A characteristic set $\mathcal {C}$ of a differential ideal 
 $\mathcal{I}$ reduces all elements of $\mathcal{I}$ to zero. Furthermore, if 
 $\mathcal{I}$ is prime, then $\mathcal{I}=\sat(\mathcal {C})$.

We can define an auto-reduced set
to be {\em irreducible} if when considered as an algebraic
ascending chain in the underlying polynomial ring, it is
irreducible. We have (\cite[p.107]{Ritt})

\begin{lem}\label{th-sat} Let $\mathcal {A}$  be an auto-reduced set. Then
a necessary and sufficient condition for $\mathcal {A}$ to be a
characteristic set of a prime differential ideal (or an irreducible differential variety) is that $\mathcal
{A}$ is irreducible. Moreover, in the case $\mathcal {A}$ is
irreducible, $\sat(\mathcal {A})$ is prime and $\mathcal {A}$ is a differential characteristic
set of it. 
\end{lem}

\begin{rem} \label{remark-algdiffcharset}
Differential characteristic sets and algebraic characteristic sets have the following obvious relation:
 Suppose $\mathcal{A}$ is a differential characteristic set of a prime differential ideal $\mathcal{I}\subset K\{x_1,\ldots,x_n\}$
 under a fixed orderly ranking $\mathscr{R}$.
 Then for any $\ell\in\mathbb{N}$, $\{\delta^{k}f: \, f\in\mathcal{A},\ord(f)\leq \ell,k\leq \ell-\ord(f)\}$ is an algebraic characteristic set of the prime ideal 
 $\mathcal{I}\cap K\{x_1,\ldots,x_n\}_{\leq \ell}$ under the ordering induced by $\mathscr{R}$.
 \end{rem}

Let $\mathcal{I}$ be a prime differential ideal in
$K \{x_1, \ldots, x_n \}$.  The \emph{differential dimension} of $\mathcal{I}$  
is defined as the differential transcendence degree of  the
differential extension field 
$\mc Q(K\{x_1, \ldots, x_n \}/\mathcal{I})$ over $K$, denoted by $\delta$-$\dim(\mathcal{I})$.
As a differential invariant, the differential dimension  of $\mathcal{I}$ can be read off from its Kolchin polynomial,
which can characterize the size of $\mathbb V(\mathcal I)$.
 
 \begin{defn}\cite{kolchin1964notion}
Let  $\mathcal{I}$ be a prime differential  ideal of $K \{x_1, \ldots, x_n \}$.
Then there exists a unique numerical polynomial $\omega_{\mathcal{I}}(t)$
such that 
$$\omega_{\mathcal{I}}(t)=\trdeg\,\mc Q \big(K \{ x_1, \ldots,x_n \}_{\leq t} / (\mathcal{I} \cap 
K \{ x_1, \ldots,x_n \}_{\leq t}) \big)/ K$$ for all
sufficiently large $t \in \mathbb{N}$.  The polynomial 
$\omega_{\mathcal{I}}(t)$ is
called the \emph{Kolchin polynomial} of $\mathcal{I}$ or its corresponding irreducible differential variety.
\end{defn}

\begin{lem} \cite[Theorem 13]{sadik2000bound} \label{def-order-sadik} 
Let $\mathcal {I}$ be a prime differential ideal  in $K \{x_1, \ldots, x_n \}$ of dimension $d$.
Then the Kolchin polynomial of $\mathcal{I}$ has the form
$\omega_{\mathcal {I}}(t)=d(t+1)+h$, where $h$ is defined to be the
\emph{order }of $\mathcal {I}$ or of $\mathbb{V}(\mathcal {I})$, that is,
$\ord(\mathcal {I})= h$.
Let $\mathcal{A}$ be a characteristic set of $\mathcal {I}$ under
any orderly ranking. Then, $\ord(\mathcal {I})= \sum_{f\in\mathcal{A}}\ord(f)$
and $d=n-\card(\mathcal A)$.
\end{lem}

 Recall that the set of numerical  polynomials can be totally ordered with respect to the ordering:
$\omega_1\leq \omega_2$ if and only if $\omega_1(s)\leq \omega_2(s)$  for all sufficiently large $s\in\m N$. 
Given a differential variety $W$,  we define a {\em generic component } of $V$ to be an irreducible component which has 
maximal Kolchin polynomial among all the components of $W$.
Lemma \ref{def-order-sadik} defines order for prime differential ideals or its corresponding irreducible differential varieties.
In this paper, we sometimes talk about the order of an arbitrary differential variety where we actually mean the order of its
generic components.

\subsection{Prolongation sequences and prolongation admissible varieties} \label{sec-prolongation}
We follow the notation of section 2 of~\cite{MPSarcs2008}. There, the authors define
a sequence of functors  $\tau_m$ indexed by the natural numbers
from varieties over $K$ to varieties over $K$ (to be honest, the functor may return 
a nonreduced scheme, but the distinction between a scheme and its reduced subscheme is 
immaterial here).   For affine space itself, 
one has $\tau_m (\m A^n) \cong \m A^{n(m+1)}$ where if we present $\m A^n$ as 
$\operatorname{Spec}(K[x_1,\ldots,x_n])$, then 
$\tau_m(\m A^n) = \operatorname{Spec}(K\{ x_1, \ldots, x_n \}_{\leq m})$.  
If $V  \subseteq   \m A^m$ is a subvariety of affine space, then 
$\tau_m V = \operatorname{Spec} \big(K \{ x_1, \ldots, x_n \}_{\leq m}\big /
( \{ \delta^j f ~:~ f \in I(V), j \leq m \} ) \big)$.  Note that the ideal 
$( \{ \delta^j f ~:~ f \in I(V), j \leq m \}  )$ is contained in 
$\I(V) \cap K \{ x_1, \ldots, x_n \}_{\leq m }$, but the inclusion 
may be proper.

 There is a natural differential algebraic map $\nabla_m:V \to \tau_m V$ given on points valued
in a differential ring by 
$$(a_1, \ldots, a_n) \mapsto (a_1,\ldots,a_n; \delta(a_1), \ldots, \delta(a_n); 
\ldots; \delta^m(a_1), \ldots, \delta^m(a_n)) \text{ .}$$

We call points in the image of $\nabla_m$ \emph{differential points}.

The image of the map $\nabla_m$ need not be Zariski dense, even on $\mc U$-valued points.   
For any differential subvariety $W  \subseteq   \mathbb A^n$, we define 
$B_m(W)$ to be the Zariski closure in $\tau_m \mathbb A^n$ of $\nabla_m (W)$.       

The functors $\tau_m$ form a projective system with the 
natural transformation $\pi_{m,\ell}:\tau_m \to \tau_\ell$ for $\ell \leq m$ given 
by projecting onto the coordinates corresponding to the first $\ell$ derivatives. 
We write $\pi_{m,\ell}:\tau_m V \to \tau_\ell V$ rather than $\pi_{m,\ell}^V$. 
Moreover, $\tau_0$ is simply the identity functor so that we write $V$ rather than
$\tau_0 (V)$.    From the definition, for $W  \subseteq   \mathbb A^n$ a differential 
subvariety, it is clear that $\pi_{m,\ell}$ restricts to make the sequence of 
varieties $(B_m(W))_{m=0}^\infty$ into a projective system of algebraic varieties in which each 
map in the system is dominant. 
 
We write $\tau^m$ for the result of composing the functor $\tau_1$ with itself
$m$ times.  There is a natural transformation $\rho_m:\tau_m \to \tau^m$ which 
for any algebraic variety $V$ gives a closed embedding $\rho_m:\tau_m V
\hookrightarrow \tau^m V$.   To ease notation, let us write the map $\rho$ in 
coordinates only for the case of $V = {\mathbb A}^1$ and $m = 3$.
The general case requires one to decorate the variables with further subscripts and to 
nest the coordinates more deeply.   Here 
$$\rho_3(x^{(0)},x^{(1)},x^{(2)},x^{(3)}) = (( ( x^{(0)}, x^{(1)}), (x^{(1)}, x^{(2)})), 
((x^{(1)},x^{(2)}), (x^{(2)}, x^{(3)})))$$

For the formal definitions of the items discussed above,  please refer to~\cite{MSJETS, MPSarcs2008}.

\begin{defn}  \label{def-pro-seq} A sequence of varieties $X_\ell  \subseteq   \tau_\ell \m A^{n}\,(\ell\in\m N)$ is called a \emph{prolongation sequence} if \begin{enumerate} 
\item The map $X_{\ell+1} \rightarrow X_\ell$ induced by the projection map $\pi _{\ell +1, \ell } : \tau _{\ell +1 } \m A^n \rightarrow \tau _ \ell \m A^n $ is dominant. 
\item For all $\ell$, $\rho_{\ell+1} (X_{\ell+1})$ is a closed 
subvariety of $\tau_1 (\rho_\ell (X_\ell))$. 
\end{enumerate} 
\end{defn}

Given a sequence of algebraic varieties $(X_\ell)_{\ell \geq 0}$ with $X_\ell \subseteq   \tau_\ell \m A^{n}$, the  
differential variety $V$ corresponding to $(X_\ell)_{\ell\geq0}$ is  
 $$V=\{b \in {\mathbb A}^n: (\forall \ell)~~\nabla_\ell (b) \in X_\ell\} \text{ .}$$
From the point of view of differential ideals, if 
$\mathcal{I} = \bigcup_{\ell = 0}^\infty I(X_\ell)  \subseteq   K \{ x_1, \ldots, x_n \}$, then 
$V = \V(\mathcal{I})$. 

By Definition \ref{def-pro-seq}, a sequence of varieties $(X_\ell\subset\tau_\ell\mathbb{A}^n)_{\ell \geq 0}$ is a prolongation sequence if and only if for each $\ell \geq 0$,
$X_\ell=B_\ell(V)$ where $V$ is the differential variety corresponding to $(X_\ell)_{\ell \geq 0}$.

There is a bijective correspondence between irreducible prolongation sequences (by which we mean each variety in the sequence is irreducible) and affine irreducible differential varieties. Given a differential variety $V$, the prolongation sequence corresponding to $V$ is given by $X_\ell=B_\ell(V)$ for all $\ell \geq 0$ ~\cite[the discussion preceding Definition 2.8]{MPSarcs2008}.
Thus, prolongation sequences are in one-to-one correspondence with affine differential algebraic varieties, and the Noetherianity of the Kolchin topology guarantees that a finite portion of a prolongation sequence determines the entire sequence.

The varieties which appear in a prolongation sequence are interesting to us.
Below, we first study the irreducible ones and define prolongation admissibility for irreducible varieties. 

\begin{defn} Given an irreducible algebraic variety $V  \subseteq    \tau_\ell \m A^ {n}$, 
we say that $V$ is prolongation admissible 
if $\rho_\ell ( V)  \subseteq   \tau  ( \rho_{\ell - 1} (\pi_{\ell,\ell-1} (V)))$.
\end{defn} 

For  irreducible prolongation admissible varieties, it is easy to establish the stronger fact that $\rho_\ell ( V)  \subseteq   \tau ^ {\ell-d} ( \rho_d (\pi_{\ell,d} (V)))$ 
for all $0 \leq d < \ell$.   Indeed, we have natural transformations $\rho_\ell \to \rho_{\ell-d} \pi_{\ell,\ell-d}$ for $0  < d   \leq \ell$ given by forgetting the coordinates
corresponding to the $k^\text{th}$ derivatives for $\ell -d < k \leq \ell$.  Applying $\tau$ and precomposing with $\pi_{\ell,\ell-1}$ to $\rho_{\ell-1} \to \rho_{\ell-d} \pi_{\ell-1,\ell-d}$
we obtain natural transformations $\tau \rho_{\ell-1} \pi_{\ell, \ell - 1} \to \tau \rho_{\ell - d} \pi_{\ell, \ell - d}$.   Thus, from an inclusion
$\rho_\ell(V)   \subseteq    \tau \rho_{\ell-1} \pi_{\ell, \ell-1}(V)$ we deduce that we must have an inclusion $\rho_{\ell-d + 1} \pi_{\ell,\ell-d+1}(V)  \subseteq   
\tau \rho_{\ell - d} \pi_{\ell, \ell -d} (V)$ as we see from the following diagram:

$$\xymatrixcolsep{7pc}\xymatrix{ 
\rho_{\ell} (V)  \ar[d]  \ar@{^{(}->}[r]  &  \tau \rho_{\ell-1} (\pi_{\ell,\ell-1} (V))   \ar[d] \\
\rho_{\ell-d+1} (\pi_{\ell,\ell-d+1} (V)) \ar@{^{(}.>}[r]  &  \tau \rho_{\ell-d} (\pi_{\ell,\ell-d} (V))
 \text{.}  \\
}$$ 

Using the fact that $\tau$ preserves inclusion, and iterating, we have 

$$\rho_\ell  (V)  \subseteq   \tau \rho_{\ell -1} (\pi_{\ell, \ell-1} (V))   \subseteq   \tau^2 \rho_{\ell -2} (\pi_{\ell, \ell -2} (V)) 
 \subseteq   \cdots  \subseteq   \tau^d \rho_{\ell- d} (\pi_{\ell, \ell-d} (V))$$
 as claimed.

 The following fact is the basis of the well known geometric axioms for differentially closed fields, written in our language: 

\begin{fact} \label{geoax}    
Given irreducible varieties $V$ and $W$ of $\mathbb {A}^n(\mc U)$ with $W  \subseteq   \tau_1 (V)$ 
so that the restriction of $\pi_{1,0}$ to $W$ is a dominant map to $V$, then 
for any $U  \subseteq   W(\mc U)$, a nonempty Zariski open set, there is $a \in V(\mc U)$
such that $\nabla_1(a) \in U$. 
\end{fact} 

Indeed, Fact~\ref{geoax} characterizes differentially closed fields amongst algebraically closed
differential fields of characteristic zero; for details, 
see~\cite{PiercePillay}. 

\begin{lem} 
\label{highorder} 
Suppose that $V  \subseteq   \tau_\ell(\m A^n)$ is an irreducible prolongation admissible variety. 
Then for any nonempty open subset $U  \subseteq   V$, there is some $a \in \m A^n (\mc U)$ 
such that $\nabla_\ell (a) \in U$. 
\end{lem}

\begin{proof}    Since $V$ is prolongation admissible,   $\rho_\ell(V)  \subseteq   \tau_1(\rho_{\ell-1} (\pi_{\ell,\ell-1}(V)))$. 
So, we have the following commutative diagram: 

$$\xymatrixcolsep{7pc}\xymatrix{ 
V  \ar[d] _{ \pi_{\ell,\ell-1}} \ar@{^{(}->}[r] ^{\rho_\ell} & \rho_\ell (V)  \subseteq   \tau_1(\rho_{\ell-1} (\pi_{\ell,\ell-1} (V))) \ar[d] ^{ \pi_{1,0}}  \\
\pi_{\ell,\ell-1} (V) \ar@{^{(}->}[r]_{\rho_{\ell-1}} & \rho_{\ell-1} (\pi_{\ell,\ell-1} (V))
 \text{.}  \\
}$$ 

Since each $\rho$ is an embedding,  we must have that the restriction of $\pi_{1,0}$ maps $\rho_\ell (V)$ dominantly to 
$\rho_{\ell-1} (\pi_{\ell,\ell-1}(V))$. 
Thus, by Fact~\ref{geoax}, the set of points 
$$\{ \nabla_1(a)| a \in \rho_{\ell-1} (\pi_{\ell,\ell-1}(V)), 
\, \, \nabla_1 (a) \in \rho_\ell (V)\}$$ is Zariski dense in $\rho_\ell(V)$,
 so the set  
 $$\rho_\ell^{-1} (\{ \nabla_1(a)| a \in \rho_{\ell-1} (\pi_{\ell,\ell-1}(V)), 
\, \, \nabla_1 (a) \in \rho_\ell (V)\})$$ is Zariski dense in $V$. 
Every such point has the form $\nabla_\ell (b)$ for $b \in \m A^n$, proving the claim. 
\end{proof} 

\begin{rem} \label{rk-pro-admissible}
From the proof of Lemma \ref{highorder},  an irreducible algebraic  
variety is prolongation admissible if and only if the differential points form a dense subset.
For an arbitrary algebraic variety, we define it to be {\em prolongation admissible} if all of its irreducible components are prolongation admissible. 
In other words, an algebraic variety is prolongation admissible if and only if the differential points form a dense subset.
So every variety in a prolongation sequence is prolongation admissible.
Furthermore, prolongation admissible varieties are precisely the varieties which appear as elements of a prolongation sequence. 
\end{rem}

 Given two prolongation sequences $X:=(X_\ell\subset\tau_\ell\mathbb{A}^n)_{\ell=0}^{\infty}$ and $Y:=(Y_\ell\subset\tau_\ell\mathbb{A}^n)_{\ell=0}^{\infty}$, define $X\leq Y$ if and only if $X_\ell \subseteq   Y_\ell$ for each $\ell\geq0$. With this binary relation, the set of all prolongation sequences forms a partially order set.
 From the definition, it is easy to see that if  $\mc V$ is a set of prolongation sequences
$X:= (X_\ell  \subseteq   \tau_\ell {\mathbb A}^n)_{\ell=0}^\infty$,  
then the sequence $( \overline{ \bigcup_{X \in \mc V}  X_\ell} )_{\ell=0}^\infty$
is also a prolongation sequence. This justifies the following definition. 

\begin{defn} Given a variety $V  \subseteq   \tau_h (\m A^n)$,
the prolongation sequence generated by $V$ is the maximal prolongation sequence  
$(V_\ell  \subseteq   \tau_\ell \m A^{n})_{\ell\geq0}$
such that  $V_h  \subseteq   V$.  
\end{defn}

%Note that the prolongation sequence generated by a variety $V  \subseteq   \tau_h (\m A^n)$ is  $(B_\ell(W)  \subseteq   \tau_\ell \m A^{n})_{\ell\geq0}$,
%where $W=\{a\in \mathbb A^n: \nabla_h(a)\in V\}$.
%By Remark \ref{rk-pro-admissible}, if $V$ is prolongation admissible, then the differential points of $V$ form a dense subset,
%so $V=B_h(W)$.

The following lemma follows from the observation above that 
the closure of an arbitrary union of prolongation admissible varieties is also 
prolongation admissible.

\begin{lem}\label{fincomp} Given $V  \subseteq   \tau_h \m A^{n}$,
there is a finite set of irreducible maximal prolongation admissible subvarieties of $V$. 
\end{lem} 
\proof 
Let $\mathcal{W}$ be the set of all prolongation admissible subvarieties of $V$ and set $W=\overline{\cup\mathcal{W}}$.
  Then clearly,  $\overline{\{\nabla(a):\nabla(a)\in W\}}=W$.    
  Let $W=\cup_{i=1}^mW_i$ be the irredundant irreducible decomposition of $W$.
  Since $\overline{\cup_{i=1}^m\{\nabla(a):\nabla(a)\in W_i\}}=W$, $\overline{\{\nabla(a):\nabla(a)\in W_i\}}=W_i$ for each $i$.
  Thus, $W_i$ is prolongation admissible and these $W_i$ are the only maximal irreducible prolongation admissible subvarieties of $V$.
\qed

\begin{lem} \label{proladmiss} 
Given a prolongation admissible variety $V \subseteq   \tau_h (\m A^n) = \m A^{n(h+1)}$, the prolongation sequence generated by $V$, denoted by $(V_i)_{ i\in \m N}$, has the property that $V_h = V$ and for each $i$, $$V_i = \overline{\{ ( a , \delta  a , \ldots , \delta ^i  a ) \, | \,  a \in \mc U, \, ( a , \delta  a , \ldots \delta ^h  a ) \in V \}}.$$\end{lem} 
\begin{proof}
It follows from  Definition \ref{def-pro-seq}  that $(V_i) _{i \in \m N}$ forms a prolongation sequence. 
It is easy to see that $(V_i) _{i \in \m N}$  is maximal among all prolongation sequences $(W_i)_{i\in\m N}$ with $W_h  \subseteq   V.$ 
And by Remark \ref{rk-pro-admissible}, $V_h = V$.
\end{proof} 

\vskip2pt 

Irreducible prolongation admissible varieties are of special interest in this paper.
The following lemma shows that the algebraic characteristic sets of irreducible prolongation admissible varieties have a special form. 
\begin{lem} \label{lem-charsetprolongationadm}
Let $V\subset\tau_h\m A^{n}$ be an irreducible prolongation admissible variety and $\mathcal{A}$ a characteristic set of $V$ under the standard orderly ranking. 
Rewrite ${\mathcal{A}}$  in the following form
  \begin{equation} \label{eq-charproadm}
  {\mathcal{A}}=\begin{array}{lll} \langle A_{10}, &\ldots,&A_{1 \ell_1}\\ \ldots&\ldots&\ldots\\ A_{p0}, &\ldots,&A_{p \ell_p} \rangle \\ \end{array}\end{equation}
  such that $\lv(A_{ij})=x_{\sigma(i)}^{(o_{ij})}\,(j=0,\ldots,\ell_i)$ and $o_{i0}<o_{i1}<\cdots<o_{i \ell_i}$.
  Then for each $i,$ $\ell_i=h-o_{i0}$, and  if $\ell_i>1$, then for all $k=1,\ldots, \ell_i$, $o_{ik}=o_{i0}+k$ and $A_{ik}$ is linear in $x_{\sigma(i)}^{(o_{i0}+k)}$.
  Moreover, $\mathcal{A}$ is a consistent differential system.
\end{lem}
 
 \proof  Let $W$ be the differential variety associated to the prolongation sequence generated by $V$.
 Suppose $W=\bigcup_{i}W_i$ is an irredundant irreducible decomposition of $W$.
 Then $B_h(W)=\bigcup_{i}B_h(W_i)=V$.
 Since $V$ is irreducible, there exists $i_0$ such that $B_h(W_{i_0})=V$. 
 Let $\mathcal{B}^\delta:= \langle B_{1},B_2,\ldots,B_m \rangle$ be a differential characteristic set of $W_{i_0}$ and suppose for $i\leq p$, $s_i=\ord(B_i)\leq h$
 and for $i>p$, $\ord(B_i)>h$. Then by Remark \ref{remark-algdiffcharset}, $\langle B_1,B_1',\ldots,B_1^{(h-s_1)},\ldots,B_p,B_p',\ldots,B_p^{(h-s_p)} \rangle$ is an algebraic characteristic set of $V$.
 Since any two characteristic sets of $V$  have the same rank,
  $\mathcal{A}$ has the desired form described as above.
 \qed

% \begin{lem} \label{monotonedim}
%  Let $(V_i)_{i\in\mathbb N}$ be an irreducible prolongation sequence. 
%  For all $k>j$, $$\dim (V_{k+1} ) - \dim (V_k ) \leq \dim (V_{j+1}) - \dim (V_j).$$ 
%\end{lem} 
%\begin{proof} 
%Let $W$ be the irreducible differential variety corresponding to the prolongation sequence $(V_i)_{i\in\mathbb N}$.
%By the saturation of $\mc U$, we can suppose $ a = (a_1,\ldots,a_n) \in \mc U^n$ is a generic point of $W$.
%Then for all $k$,  $\nabla _{k} ( a)$ is generic point of $B_k(W)=V_k$. 
%Let $\{a_{i_\ell}^{(j+1)}\} _{\ell=1} ^d$ be a transcendence base for $ a ^{(j+1)}=(a_1^{(j+1)},\ldots,a_n^{(j+1)})$ over $K ( \nabla _j (  a))$. Then $\{a_{i_\ell}^{(k+1)}\} _{\ell=1} ^d$ contains a transcendence base for $ a ^{(k+1)}$ over $K ( \nabla _k (  a))$ and the result follows. 
%\end{proof} 

\begin{defn}
Let $V$ be an irreducible prolongation admissible variety of $\tau _h (\m A^n)$
and $W$ be the differential variety corresponding to the prolongation sequence generated by $V$. 
A component $W_1$ of $W$ is called a {\em dominant} component if it satisfies $B_h(W_1)=V$. 
\end{defn}

\begin{lem} \label{uniquegenericomponent}
 Let $V$ be an irreducible prolongation admissible subvariety of $\tau _h (\m A^n)$. 
 Let $W$ be the differential variety corresponding to the prolongation sequence generated by $V$. 
Then there is a unique dominant component. 
Moreover, the Kolchin polynomial of the dominant component $W_1$ has the form $$\omega_{W_1}(t)=\big(\dim(V)-\dim(\overline{\pi_{h,h-1}(V)})\big)(t-h)+\dim(V),$$
and for $t\geq h-1$, $\dim(B_t(W_1))=\omega_{W_1}(t)$.\end{lem} 

\proof 
 Since $V$ is prolongation admissible, by Lemma~\ref{lem-charsetprolongationadm},
 $V$ has a characteristic set $\mathcal{A}$ of the form  
 \begin{equation} \label{eq-charset} \nonumber
  {\mathcal{A}}=\begin{array}{lll} \langle A_{10}, &\ldots,&A_{1 \ell_1}\\ \ldots&\ldots&\ldots\\ A_{p0}, &\ldots,&A_{p \ell_p} \rangle \\ \end{array}.\end{equation}
  satisfying the corresponding property in Lemma~\ref{lem-charsetprolongationadm}.

 Let $\mathcal{B}^\delta= \langle A_{10}, A_{20}, \ldots, A_{p0} \rangle$.
 Then it is clear that $\mathcal{B}^\delta$ is a differential auto-reduced set which is irreducible.
 Then by Lemmas \ref{th-sat} and \ref{def-order-sadik}, $T=\V(\sat(\mathcal{B}^\delta))$ is an irreducible differential variety of differential dimension $n-p$.
 We first claim that $B_h(T)=V$. Indeed, since $V$ is prolongation admissible, $I(V)=\mathbb{I}(W)\cap K[(x_{i}^{(j)})_{1\leq i\leq n;j\leq h}]$.
 So $A_{i0}\in \mathbb{I}(W)$ implies  $A_{i0}^{(\ell)}\in I(V)$ for $\ell\leq h-\ord(A_{i0})$.
 Let  $\mathcal{C}:=\langle A_{10},\ldots, A_{10}^{(h-\ord(A_{10}))},\ldots,A_{p0},\ldots, A_{p0}^{(h-\ord(A_{p0}))} \rangle$.
 Then we have $\asat(\mathcal{C})=\sat(\mathcal{B}^\delta)\cap K[(x_{i}^{(j)})_{1\leq i\leq n;j\leq h}]$ and correspondingly, $B_h(T)=V(\asat(\mathcal{C}))$. 
 On the other hand, since $\mathcal{C}\subset I(V)=\asat(\mathcal{A})$ and the auto-reduced property of $\mathcal{A}$ implies that the initial and separant of $A_{i0}$ are all reduced with respect to $\mathcal{A}$,  $\asat(\mathcal{C}) \subseteq   \asat(\mathcal{A})$.
Note that $\asat(\mathcal{C})$ and $\asat(\mathcal{A})$  are prime ideals of the same dimension,
so we have $\asat(\mathcal{C})=\asat(\mathcal{A})$.
 Thus,  $B_h(T)=V(\asat(\mathcal{A}))=V$.

 We now show that $T \subseteq   W$. Let ${b}\in\mathbb A^n$ be a generic point of $T$.
 By Lemma  \ref{proladmiss}, the prolongation sequence $(V_i)_{ i\in \m N}$ generated by $V$  has the property that  for each $i$, 
 $$V_i = \overline{\{ ( a , \delta  a , \ldots , \delta ^i  a ) \, | \,  a \in \mc U, \, ( a , \delta  a , \ldots \delta ^h  a ) \in V \}}.$$
 So $W=\{a \in \mathbb A^n: \nabla_i(a)\in V_i, i\in\mathbb{N}\}$.
 Since $B_h(T)=V$, $\nabla_h(b)\in V$, $b \in W$. So $T \subseteq   W$.
  It remains to show that $T$ is the unique dominant component of $W$.

Let $W=\bigcup_iW_i$ be the irredundant irreducible decomposition of $W$.
 Since $B_h(W)=V$, $B_h(W_i) \subseteq   V$ for each $i$ and  there exists at least one  component $W_i$ satisfying $B_h(W_i)=V$, that is, a dominant component.
 Now suppose $W_{i_0}$ is an arbitrary dominant component of $W$.  
 Then $B_h(W_{i_0})=V$ and by Remark \ref{remark-algdiffcharset},  $\mathcal{B}^\delta$ may be  a subset of a differential characteristic set of  $W_{i_0}$.
So $\sat(\mathcal{B}^\delta)\subseteq \m I(W_{i_0})$ and $W_{i_0}\subseteq T$ follows. 
Since $T \subseteq   W$, there exists $i_1$ such that $T\subseteq W_{i_1}$.
From $W_{i_0}\subseteq T\subseteq W_{i_1}$, we have $W_{i_0}=T= W_{i_1}$. 
 Thus, $T$ is the unique dominant component of $W$.

For the second assertion, since the order of the polynomials in a characteristic set of $T$ under an orderly ranking 
is bounded by $h$,
 $\omega_T(t)=\dim(B_t(T))$  for $t\geq h-1$. 
 Since $\omega_T(t)=\delta$-$\dim(T)(t+1)+\ord(T)$, $B_h(T)=V$ and $B_{h-1}(T)=\pi_{h,h-1}(V)$,
 $\omega_T(t)$ has the desired form.
\qed

\begin{rem} \label{rem-uniquegenericomponent}
Given an irreducible prolongation admissible variety $V \subseteq   \tau_h(\m A^n)$,
 Lemma \ref{uniquegenericomponent} tells us that the differential variety $W$ corresponding to the prolongation sequence generated by  $V$ has a unique dominant component, $W_1$.
Apart from the above,  Lemma \ref{uniquegenericomponent} and its proof give us more information of  $W_1$,
which we would like to point out here to make things clearer.

First, both the differential dimension and order of $W_1$ can be computed from $V$.
To be more specific,  the differential dimension of $W_1$ is equal to $d=\dim(V)-\dim(\overline{\pi_{h,h-1}(V)})$,
 and the order of $W_1$ is equal to $\dim(V)-d(h+1)$.

Second, $W_1$ can be recovered from $V$ in terms of characteristic sets:
Let $\mathcal{A}:= \langle A_{10},A_{11},\ldots,A_{1\ell_1},\ldots,A_{p0},A_{p1},\ldots,A_{p\ell_p} \rangle$ be
an algebraic characteristic set of $V$ arranged as in the form (\ref{eq-charset}).
Then $\mathcal{B}:= \langle A_{10},A_{20},\ldots,A_{p0} \rangle$ is a differential characteristic set of $W_1$,
 and $\sat(\mathcal{B})$ is the defining differential ideal of $W_1$.

In addition, as we will see in Section \ref{Chowsection}, the differential Chow form of $W_1$ and the index of $W_1$ can be computed from $V$.
The previous result has a partial differential analog, but the situation is more complicated. See~\cite{FLL2015} for details. 
\end{rem}

We conclude this section by giving  simple examples to illustrate  prolongation admissible varieties as well as results in Lemma \ref{uniquegenericomponent} and Remark \ref{rem-uniquegenericomponent}.

\begin{example} \label{ex-1}
Let $V\subset\tau_2\mathbb A^2$ be the algebraic variety defined by $x_1'=0$ and $x_2''=0.$
Then the prolongation sequence generated by $V$ is $(V_\ell\subseteq\tau_\ell\mathbb{A}^2)_{\ell=0}^\infty$ where $V_0=\mathbb A^2$,
$V_1=V(x_1')\subset \tau_1\mathbb{A}^2$, $V_2=V(x_1',x_1'', x_2'')\subset\tau_2\mathbb{A}^2$ and $V_\ell=V(x_1',x_1'',\ldots,x_1^{(\ell)},x_2'',\dots,$ $x_2^{(\ell)}) \subseteq  \tau_\ell\mathbb{A}^2\,(\ell\geq 3)$.
Since $V_2\subsetneq V$, $V$ is not a prolongation admissible variety.

Let $U=V_2=V(x_1',x_1'', x_2'')\subset\tau_2\mathbb A^2$.
Then as illustrated above, $U$ is a prolongation admissible variety and the prolongation sequence generated by $U$ is just the same $(V_\ell\subseteq\tau_\ell\mathbb{A}^2)_{\ell=0}^\infty$ as above.
The differential variety corresponding to $(V_\ell\subseteq\tau_\ell\mathbb{A}^2)_{\ell=0}^\infty$ is $W=\mathbb V(x_1',x_2'')$, which is irreducible, hence also the dominant component of itself. Note that $U=B_2(W)$.
\end{example}

The following examples show how each  irreducible differential variety $W$ of order $h$ can be determined by the corresponding prolongation admissible variety $B_h(W)$.

\begin{example} \label{ex-3}
Let $W=\mathbb V(x_1',x_2'')\subset\mathbb A^2$ be the irreducible differential variety defined by $x_1'=0, x_2''=0$. 
Clearly, the  order of $W$ is $h=3$ and $V=B_3(W)=V(x_1',x_1'', x_1^{(3)},x_2'', x_2^{(3)})\subset\tau_3\mathbb A^2$ is prolongation admissible.
The prolongation sequence generated by $V$ is the same $(V_\ell)_{\ell=0}^\infty$ as in Example \ref{ex-1}.
As illustrated in Example \ref{ex-1}, then the differential variety corresponding to $(V_\ell)_{\ell=0}^\infty$ is $W$, which is the unique dominant component of itself.
\end{example}

\begin{example} \label{ex-2}
Let $W=\mathbb V(x'^2-4x, x''-2)\subset\mathbb A^1$. 
It is clear that $W$ is an irreducible differential variety of  order $h=1$.
Moreover, $V=B_1(W)=V(x'^2-4x)\subset\tau_1\mathbb A^1$ is prolongation admissible and $\mathcal{A}:=x'^2-4x$ is an algebraic characteristic set of $V$ under the ordering $x<x'$. 
The prolongation sequence generated by $V$ is $(V_\ell)_{\ell=0}^\infty$ with $V_0=\mathbb{A}^1$,
$V_1=V$, $V_2=V(x'^2-4x, x'(x''-2))\subset\tau_2\mathbb{A}^1$ and $V_\ell=V([x'^2-4x]\cap K[x,x',\ldots,x^{(\ell)}])\subset\tau_\ell\mathbb{A}^1\,(\ell\geq 3)$.
Then the differential variety corresponding to $(V_\ell)_{\ell=0}^\infty$ is $X=\mathbb V(x'^2-4x)$, which has two components, 
the general component $W$ in Ritt's sense and the singular component $\mathbb V(x)$. 
Clearly, $W$ is the unique dominant component of $X$, which happens to be a generic component of $X$.
Moreover, from the proof of  Lemma \ref{uniquegenericomponent}, $\mathcal{B}^\delta:= \langle x'^2-4x \rangle$ is a differential characteristic set of $W$,
so $W=\sat(\mathcal{B}^\delta)$ can be determined by $V=B_1(W)$, which coincides with the classical result.
\end{example}

Generic components of a differential variety are components with maximal Kolchin polynomial.
One might think that the differential variety associated to an irreducible prolongation admissible variety also has a unique generic component and the generic component is also the dominant one. Unfortunately, this is invalid.
The following example shows how generic components and dominant components may differ. 
 \begin{example} \label{ex-3}
Let $n=3$, $f_1=y_1''^2-4y_1', f_2=y_1'y_2''+y_2^2-1,$ and $f_3=y_1y_3''+y_2^2-1$.
Let $V={V}(f_1,f_2,f_3)\subset\tau_2 \m A^3$. 
  With the help of Maple, we know $I=(f_1,f_2,f_3)\subset K[y_i^{(j)}:j\leq 2]$ is a prime ideal of dimension 6.
  So $\sat(f_1,f_2,f_3)\cap K[(y_i^{(j)}):j\leq 3]=I$  and $V=V(I)$ is irreducible and prolongation admissible.
 By~Remark \ref{rem-uniquegenericomponent},
  the   differential variety $W={\mathbb V}(f_1,f_2,f_3)\subset\m A^3$ has a unique dominant component $\mathbb{V}\big(\sat(f_1,f_2,f_3)\big)$.
  
  Performing the Rosenfeld-Groebner algorithm, one can show that $W$ has 
  two generic components, $W_1={\mathbb V}(y_1,y_2-1)$ and $W_2={\mathbb V}(y_1,y_2+1)$. 
 Clearly, both $B_2(W_1)$ and  $B_2(W_2)$ are properly contained in $V$, 
 so both of them  are not dominant components.
\end{example}

\subsection{Definability in the theory of  DCF$_0$ and Ritt problem}

We shall speak of definable families of definable sets and of certain properties
being definable in families.   These are general notions but we shall use them 
only for the theories of algebraically closed fields of characteristic zero
and of differentially closed fields of characteristic zero.   In these cases,
``definable'' is synonymous with ``constructible'' or ``differentially constructible'',
respectively.

\begin{defn}
We say that a 
family of sets $\{ X_{a} \}_{a \in B}$ is a \emph{definable family} if there
are formulae $\psi (x; y)$ and $\theta(y)$ so that $B$ is the set of realizations 
of $\theta$ and for each $a \in B$,  $X_{a}$ is the set of realizations 
of $\psi(x;a)$.   

Given a property $\mc P$ of definable sets, we say that $\mc P$ is definable in families if for any family of definable sets $\{ X_a \}_{a \in B}$ 
given by the formulae $\psi(x;y)$ and 
$\theta(y)$, there is a formula $\phi(y)$ so that the set 
$\{ a \in B ~:~ X_a \text{ has property } \mc P \}$ is defined by $\phi$. 

Given an operation $\mc F$ which takes a set and returns another set, we 
say that $\mc F$ is definable in families if for any  
family of definable sets $\{ X_a \}_{a \in B}$  
given by the formulae $\psi(x;y)$ and 
$\theta(y)$, there is formula $\phi(z;y)$ so that for each $a \in B$, 
the set $\mc F (X_a)$ is defined by $\phi(z;a)$. 
\end{defn} 

Of particular importance for us will be several specific incarnations of $\mc F$ from the previous definition. Consider the definable 
family $\{ X_a \}_{a \in B}$ (which we assume to be definable in some theory expanding the theory of fields). If $\mc F$ takes 
the set of points $X_a$ and returns the Zariski closure of the set over $K$, then proving that $\mc F$ is definable in families in a 
given theory amounts to establishing a bound for the degree of the polynomials which define the Zariski closure of $X_a$ which 
depends on the formula defining $X_a$, but is independent of the element $a$. In the theory of differentially closed fields, the 
main results of \cite{JOnfcp} establish results of this form. 

Given a definable family $\{ X_a \}_{a \in B}$   of Zariski closed sets over some field $K$, another important example occurs 
when $\mc F$ takes the closed set $X_a  \subseteq  \m A^n$ and returns the set of components of $X_a$.  
In this case, we need to take some care in presenting $\mc F$ as an operation.  Strictly speaking, ${\mc F}(X_a)$ is 
a finite set of algebraic subvarieties of $\m A^n$.    In general, if $\{ X_a \}_{a \in B}$ is a constructible family of 
(possibly reducible) varieties, then there is a bound $N$ on the number of irreducible components of $X_a$ depending
just on the family.  Thus, ${\mc F}(X_a)$ may be presented as a sequence of subvarieties of $\m A^n$ of 
length at most $N$, up to reordering.  Using the theory of symmetric polynomials (or elimination of 
imaginaries relative to the theory of algebraically closed fields), such a finite set (or a finite sequence up to permutation) 
may be represented by a finite sequence.       

We will require the following facts about definability in algebraically closed fields. 

\begin{fact} \label{factdef} We work relative to the theory of algebraically 
closed fields (ACF).   
\begin{enumerate} 
\item \label{clodef} The Zariski closure is definable in families. 
\item \label{dimdegdef} The dimension and degree of the Zariski 
closure of a set are definable in families. 
\item \label{irrdef} Irreducibility of the Zariski closure is a definable property. 
More generally, the number of components of the Zariski closure
is definable in families. 
\item \label{hypdegvardef} If the Zariski closure is an irreducible hypersurface given by the vanishing of some nonzero polynomial, then the degree of that
polynomial in any particular variable is definable in families. 
\item \label{compmap} The family of irreducible components of the 
Zariski closure is definable in families. 
\end{enumerate} 
\end{fact}

Fact~\ref{factdef} is established in the Appendix~\ref{mainapp}. As we 
noted in the introduction, other proofs appear in the literature. 
\subsubsection{Definability in the theory of differentially closed fields}
%
%As we described in the previous section, the facts corresponding to \ref{factdef} are not known in the differential setting. Our replacement for these facts will be the results of this section and the use of prolongation sequences. 

We now return to differential fields and develop results about definability with respect to the theory of differentially closed fields in this section.
In particular, we will show how far we can go for the differential analogs of results in Fact \ref{factdef}.
In the first place, we show that the differential dimension and order are definable in families.

In~\cite{GDIT}, intersections of differential varieties with 
generic differential hyperplanes were analyzed. 
The coefficients of the
defining equation of a generic differential hyperplane, $u_{0}+u_{1}x_1+\cdots+u_{n}x_n=0$,  are taken to be  differential indeterminates  over the 
differential field $K$ over which the variety is defined and various aspects of the 
geometry of the resulting intersection is established over the field $K\langle u_0,\ldots,u_n\rangle$.   
In particular, the following result was proved.

\begin{thm} \label{GIT} \cite{JBertini, GDIT} Let $V  \subseteq   \m A^n$
be an irreducible 
affine differential variety of differential dimension $d$ and order $h$. 
Let $H$ be a generic differential hyperplane defined by a linear form $u_{0}+u_{1}x_1+\cdots+u_{n}x_n$,
whose coefficients $u_{i}$ are differentially independent over $K$.
 Then over $K\langle u_0,\ldots,u_n\rangle$,
 $V \cap H$ is nonempty if and only if $d>0$.
In the case $d>0$, $V \cap H$ is  an irreducible differential variety of differential dimension $d-1$ and order $h$.
\end{thm} 

One can use Theorem~\ref{GIT} to prove the definability of dimension and order; as we have remarked above, there seem to be various other ways to prove these results.

\begin{lem} \label{semicont} Given a differentially  constructible 
family of differential varieties $(X_s)_{s \in S (\mc U)}$, with $\dim (S)=0$,
the set $\{s \in S~:~ \dim (X_s) = d \}$ is a  
differentially constructible subset of $S$ 
\end{lem} 
\begin{proof} 
Fix $d+1$   
tuples $(c_{i,j})_{1 \leq i \leq d+1 , 0 \leq j \leq n}$  of length $n+1$ such that all these $c_{i,j}$ are differentially independent over $K$.
  Then by Theorem~\ref{GIT}, for $s \in S$,
    $ \dim (X_s) =d$ is equivalent to the condition
$$X_s \cap  \mathbb V \big(\{ c_{i,0}+ \sum_{j=1}^n c_{i,j} y_j \}_{i=1}^d\big) \neq \emptyset \text{ and } X_s \cap  
\mathbb V \big(\{ c_{i,0} +\sum_{j=1}^n c_{i,j} y_j \}_{i=1}^{d+1}\big) =\emptyset. $$
%  \begin{eqnarray*}
%  \dim (X_s) \geq d & \longleftrightarrow &  
%  X_s \cap  
%  V (\{ \sum_{j=1}^n c_{i,j} y_j -c_{1,n+1} \}_{i=1}^d) \neq \emptyset 
%  \text{.}
%  \end{eqnarray*}
\end{proof} 

One should note that Theorem \ref{GIT} applies in this case only because over any base
 of $S$, we know that any point on $S$ is of differential transcendence degree $0$. So, 
choosing some collection of independent differential transcendental elements over the 
base of all of the definable sets,  the collection is independent and 
differentially transcendental  over any given point in $S$.

\begin{lem}\footnote{The proof presented here is similar to the technique used in~\cite[Section 8.6]{freitag2012model}.} \label{ddimdef} Differential dimension is definable
 in families. That is, given a differentially constructible
 family of differential varieties $(X_s)_{s \in S}$ and a number 
 $d \in \m N$ the 
 set $\{s \in S:~ \dim (X_s) = d \} $ is a 
differentially constructible subset of $S$. 
\end{lem}
 
\begin{proof}
Adopt the notation of Lemma~\ref{semicont}.  
Suppose that $\dim(S) = n_1$.
Then pick $2n_1+1$ systems of $d(n+1)$-tuples of mutually independent 
differential transcendentals 
(equivalently, in model theoretic terms, fix an indiscernible set in the generic type, over $K$; 
then pick any $(2n_1+1)d(n+1)$ elements from this set). 
Denote the chosen elements $$\{c_{k,i,j} ~:~
 1 \leq k \leq 2 n_1+1 , \, 1 \leq i \leq d, \, 0 \leq j \leq n \}$$ 
 Of course, for any $s\in S$, over $\mathbb Q\langle s\rangle$, 
 some of the $2n_1+1$ systems do not give generic independent sets of hyperplanes. 
 But, because $\dim(S)=n_1$ and the systems are mutually independent, at least 
 $n_1+1$ of the systems are generic over  
$\mathbb{Q}\langle s\rangle$ for any given $s\in S$.

Now, the requirement that 
$\dim(X_s) \geq d$ is equivalent to the condition that for at least 
$n_1+1$ values of the $k$,  
$$X_s\cap \m V(c_{k,1,0}+\sum_{j=1}^n c_{k,1,j} y_j , \ldots , c_{k,d,0}+\sum_{j=1}^n c_{k,d,j} y_j) \neq \emptyset \text{ .}$$	
\end{proof}

\begin{rem}
Lemma~\ref{ddimdef} admits another proof.   A differential constructible set 
$X \subseteq {\mathbb A}^n$ has differential dimension at least $d$ just in case there
is some coordinate projection $\pi:{\mathbb A}^n \to {\mathbb A}^d$ for which 
$\pi(X)$ is Kolchin dense in ${\mathbb A}^d$.  In general, we do not know 
how to test definably whether some constructible set is Kolchin dense in another
differential variety, but in the case of affine $d$-space, a constructible set 
is dense if and only if it is generic in Poizat's sense for the additive group structure, 
that is, finitely many additive translates of $\pi(X)$ cover ${\mathbb A}^d$.   
An easy Lascar rank computation shows that $d+1$ translates would suffice. 
\end{rem}

The order of a family of zero-dimensional differential varieties is definable in 
families \cite[Appendix A.1]{HrIt}. The general result follows by 
reducing to this case via an argument similar to the proof of Lemma~\ref{ddimdef}; 
See~\cite[section 8.6]{freitag2012model} for complete details. 

\begin{lem}\label{dorderdef} The order of a definable set is definable in families.
\end{lem} 
\begin{proof} 
	The proof is similar to the argument given in \cite[Theorem 8.6.3]{freitag2012model}. If $(X_s)_{s \in S}$ is a family of differential varieties, and $S$ has differential dimension $n_1$, then picking $2n_1 +1$ many systems of $(d+1)(n+1)$-tuples of independent differential transcendentals $\{(a_{i,j}) \, | \, i =1, \ldots , 2n_1+1, \, j =1, \ldots , (d+1)(n+1)\}$ , then for any $s \in S$, at least $n_1+1$ of $\{(a_{k,j}) \, | \,  j =1, \ldots , (d+1)(n+1)\}$ is a collection of independent differential transcendentals over $s$. Any such collection of tuples determins $d+1$-many independent generic hyperplanes in $\m A^n$ over $s$. Thus by the main theorem of \cite{JBertini}, the order of the intersection of $X_s$ with this system of hyperplanes is equal to the order of $X_s$. So, the order of $X_s$ is equal to the order given by the intersection of $X_s$ with the $d+1$ many hyperplanes for at least $n_1+1$ many choices of the hyperplane system. So, we are reduced to showing that the order of a family of zero dimensional differential varieties is definable in families.
	
	Given a zero-dimensional differential variety $X \subset \m A^n$ defined by a collection of differential equations of order bounded by $h$, such that the zero set has differential dimension zero, we may write $X$ as a $D$-variety in the sense of \cite[Proposition 1.1]{HrIt} of $\m A^{n(h+1)}$, and the order of $X$ is given by the transcendence degree of the underlying algebraic variety.  Algebraic dimension of algebraic varieties is definable in families, and thus so is the order. 
	\end{proof}

In section \ref{sec-prolongation}, prolongation admissible varieties are defined and important properties are developed.
This special class of varieties plays an important role in Section \ref{deltaChow-1}.
The next lemma shows that prolongation admissible is a definable condition, which  follows from Fact \ref{factdef} and the definition of prolongation admissible. 

\begin{lem} \label{prolongadmissible-definable}  
Prolongation admissibility is definable in families. 
\end{lem} 
\proof  Let $(V_{{b}})_{{b}\in B}$  be a definable family of algebraic varieties in $\tau_h\mathbb{A}^n$ with $V_{b}$ defined by
$f_i({b},x,x',\ldots,x^{(h)})=0, \,i=1,\ldots,m$, where $x^{(i)}=(x_1^{(i)},\ldots,x_n^{(i)})$. 
By abuse of notation, let $B_h(V_b)$  be the Zariski closure of $\{\nabla_h(\bar{a}): \nabla_h(\bar{a})\in V_b\}$ in $\tau_h\mathbb{A}^n$.  
Then $\deg(B_h(V_b))$ has a uniform bound $T$ in terms of the degree bound $D$ of $f_i$, $m$, $n$ and $h$.
Indeed, let $z_{ij}\,(i=1,\ldots,n; j=0,\ldots,h)$ be new differential variables and replace $x_i^{(j)}$ by $z_{ij}$ in each $f_i$
to get a new differential polynomial   $g_{i}$.
 Consider the new differential  system $S:=\{ g_1,\ldots,g_m, \delta(z_{i,j-1})-z_{ij}:j=1,\ldots,h\}$.
 Regard $S$ as a pure algebraic polynomial system in $z_{ij}$ and $\delta(z_{ij})$ temporarily, and let $U$ be the 
 Zariski closed set  defined by $S$ in $\tau(\tau_{h}\mathbb{A}^n)$.
 Let $Z=\{\bar{c}=(c_{10},\ldots,c_{n0},\ldots,c_{1h},\ldots,c_{nh})\in\tau_h\mathbb{A}^n:  (\bar{c},\delta(\bar{c}))\in U\}$.
  Clearly,  $Z=\{\nabla_h(\bar{a}): \nabla_h(\bar{a})\in V_b\}$.
By \cite[Remark 3.2]{HPnfcp}, the degree of the Zariski closure of $Z$, namely $B_h(V_b)$, is bounded by  $D_1=D^{m(2^{n(h+1)}-1)}$.

By \cite[Proposition 3]{Heintz}, an irreducible algebraic variety $V$ can be defined by $n(h+1)+1$ polynomials of degree bounded by the degree of $V$. So $B_h(V_b)$ can be defined by at most $(n(h+1)+1)^{D_1}$ polynomials of degree bounded by $D_1^2$.
Hence, $(B_h(V_b))_{b\in B}$ is a definable family.
Recall that $V_b$ is prolongation admissible if and only if $V_b=B_h(V_b)$,
which implies that $\{b: V_b \,\text{is  prolongation admissible}\}$ is a definable set.
Thus,  prolongation admissibility is definable in families. 
\qed

By Lemma \ref{prolongadmissible-definable} and Fact \ref{factdef}, we are safe to talk about definable families of irreducible prolongation admissible varieties.
 Given a definable family $\{ X_a \} _ {a \in B}$ of irreducible prolongation admissible varieties in $\tau_{h}\mathbb A^n$, 
 as illustrated in Remark  \ref{rem-uniquegenericomponent}, the Kolchin polynomial of the dominant component of the differential variety corresponding to the prolongation sequence generated by $X_a$ is determined by the dimensions of $X_a$ and $\pi_{h,h-1}(X_a)$. 
So the order and the differential dimension of the dominant component of the  corresponding differential variety are definable in families.
This fact will be used in Section \ref{deltaChow-1}.

\subsubsection{The Ritt problem}
To establish the differential analog of Fact \ref{factdef}, it is natural to ask whether the Kolchin closure and 
irreducibility of \emph{differential} varieties are definable in families.
However, neither of these are known to be 
definable in families. This essentially comes down 
to the fact that it might not be possible to bound the orders of the differential 
polynomials which witness the non-primality of the differential ideal 
only from geometric datas. 
Developing such a bound is equivalent to several problems considered by 
Ritt~\cite[for instance, see the statement of Theorem 5.7 along with 
the references in the 
following remark]{MoosaTrainor}, and we will refer to the development of such a bound as 
the \emph{Ritt problem}. 

Characteristic sets are \emph{an} answer to this problem; various properties 
become definable in families of characteristic sets. The difficulties associated with the Ritt 
problem are the reason that our approach in this paper uses both prolongation sequences and characteristic sets. 

The drawback of characteristic sets is that for points $p$ such that the product of the separants of a given characteristic set vanish at $p$, determining if $p$ is in the differential variety with a given characteristic set is an open problem~\cite[see the discussion in the appendices beginning on page 4286]{HrIt}. In this paper, we will parameterize characteristic sets of certain differential cycles rather than parameterizing generators of differential ideals. One might seek a more direct parameterization by generators of differential ideals, but doing so while following our general strategy would, at least on the surface, seem to require a solution to the Ritt problem. 

Here is a specific indication of the problems that can arise when working directly with the generating sets of differential ideals; the following example shows that the order, $h$, will not suffice for the sort of bound described in the previous paragraphs. 

\begin{exam}~\cite{HrIt}
Let $V = \V( 2x^{(1)} x^{(3)} - (x^{(2)})^2 - 2x)$. 
Differentiating the defining equation results in the equation 
$2x^{(1)} ( x^{(4)} -1) =0$. From this, it is easy to see that $V$ consists of 
two components, $x=0$ and the generic component.
\end{exam} 

Of course, more differentiations might be necessary:

\begin{exam}~\cite{Ritt}
Consider $V=\V(f)$ where $f=(y^{(2)})^2-y\in K\{y\}$.  
Differentiating $f$ successively $3$ times, one obtains \begin{eqnarray*} 
\delta f & = & 2y^{(2)}y^{(3)}-y^{(1)}  \\ \delta^2 f &=&2y^{(2)} y^{(4)}+
2(y^{(3)})^2-y^{(2)} \\
\delta^3 f &=&2y^{(2)} y^{(5)}+6y^{(3)}y^{(4)}-y^{(3)} \end{eqnarray*}

Then $2y^{(3)}\cdot \delta^3 f -(6y^{(4)}-1)f^{(2)} =y^{(2)} (4y^{(3)}y^{(5)}-12(y^{(4)})^2+8y^{(4)}-1)\in[f]$.
Thus, $V=\V(f, y^{(2)})\cup \V(f,4y^{(3)}y^{(5)}-12(y^{(4)})^2+8y^{(4)}-1)$ is reducible.
\end{exam}

Informally, the Ritt problem asks if there is an upper bound to the 
number of required differentiations
in terms of the ``shape" of the equations. An equivalent form of the 
Ritt problem~\cite[Appendix 1]{HrIt} is testing when a given point (say $0$) 
at which the separants of the characteristic set vanish is in the 
generic component of the differential ideal generated by the characteristic set.

\begin{rem}
Note that although many of the arguments and bounds in this paper are theoretical,  most of them could be made effective (for instance, many of the definability arguments could be made effective using differential elimination algorithms). 
\end{rem}
%\begin{rem} 
%To revisit Example \ref{ex-2},  we consider the differential variety $V=\mathbb V((x')^2 -4x)$. 
%The differential ideal $\mathcal{I}=\{(x')^2 -4x\}  \subseteq   K \{y \}$ is not prime and the minimal 
%prime components of $\mathcal{I}$  are the generic component $\sat((x')^2 -4x)$ and the singular component $[x]$. 
%Of course, $h=1$ is a bound for the order of the components of the equation (in the notation of the proof of Lemma \ref{singdef}). 
%The only maximal irreducible prolongation admissible subvariety of $B_1(V)\subset\tau_1\m A^1$ is the entire variety.
%Through the process used in the proof of Lemma~\ref{singdef}, we only get the generic component $\mathbb V(\sat((x')^2 -4x))$ of $V$,
%but can not get the singular component $\mathbb V(x)$. 
%
%
%There is a way of \emph{partially} remedying this defect given in~\cite{golubitsky2009generalized}; let $S$ be a collection of differential polynomials. 
%Then there is an algorithm which produces a finite set of characteristic sets such that the radical differential ideal $I$ generated by $S$ is the intersection of the prime differential ideals given by the characteristic sets. Then the minimal primes of the differential radical ideal generated by $S$ are among the prime differential ideals given by the finite list of characteristic sets. However, there is no known algorithm for testing which of these ideals are actually minimal. For complete details, see~\cite{golubitsky2009generalized}. 
%\end{rem} 
%

\section{Algebraic Chow forms and differential Chow forms} \label{Chowsection}
In this section we recall  the definitions of Chow forms, Chow varieties, and their 
differential algebraic analogs. The algebraic 
Chow form was first defined for  projective varieties by Chow~\cite{Chow}.
When $\sum_{i=0}^n c_i y_i$ is a linear form in $y_0,\ldots,y_n$ with coefficients $\{c_i\}_{i=0}^n$ a tuple of
independent transcendentals, we call the form algebraically generic, and we call the zero 
set of such a form a generic hyperplane.

\begin{defn}~\cite{Chow,Hodgevol2}  
Let $V  \subseteq   \m P^n$ be an irreducible projective variety of dimension $d$.
Take $d$ independent generic linear forms
$L_i=v_{i0}y_0+\cdots+v_{in}y_n$ for $1 \leq i \leq d)$, then 
$V$ intersects $V(L_1, \ldots, L_d)$ in a finite set of points,  say $(\xi_{\tau0},\ldots,\xi_{\tau n})\,(\tau=1,\ldots,m)$.
Then there exists a polynomial $A\in K[\bv_1,\ldots,\bv_d]$ such that $F(\bv_0,\ldots,\bv_d)=A\prod_{\tau=1}^m(\sum_{j=0}^nv_{0 j}\xi_{\tau j})$ is an irreducible polynomial in $K[\bv_0,\ldots,\bv_d]$ where $\bv_i=(v_{i0},v_{i1},\ldots,v_{in})$.
This $F$ is called the \emph{algebraic Chow form} of $V$. 
\end{defn} 

The Chow form $F$ is homogeneous in each $\bv_i$ of degree $m$. 
We call $m$ the \emph{degree} of $V$, denoted by $\deg(V)$. Throughout the 
remainder of this paper, unless otherwise indicated, varieties and differential 
varieties are affine. We  now introduce the concept of algebraic Chow form for 
irreducible varieties in $\m A^n$.

\begin{defn}
Let $V \subseteq  \mathbb{A}^n$ be an irreducible affine  variety of dimension $d$.
Let $V' \subseteq  \mathbb{P}^n$ be the projective closure of $V$ with respect to the usual inclusion of $\mathbb A^n$ in $\mathbb P^n$ (identifying $(a_1,\ldots,a_n)\in {\mathbb A}^n$ with $[1:a_1:\ldots:a_n]\in \mathbb  P^n$). 
We define the algebraic Chow form of $V$ to be the algebraic Chow form of $V'$. 
\end{defn}

An (effective)  \emph{algebraic cycle} in $\m A^n$ of dimension $d$ over $K$ is of the form $V=\sum_{i=1}^\ell t_ i V_ i\,(t_i\in\m Z_{\geq0})$ where each $V_i$ is an irreducible variety of dimension $d$ in $\m A^n$. 
We define the algebraic Chow form of $V$ to be $F(\bv_0,\ldots,\bv_d)=\prod_{i=1}^\ell (F_ i(\bv_0,\ldots,\bv_d))^{t_ i}$ where $F_i$ is the algebraic Chow form of $V_i$,
 and define the degree of $V$  to be $\sum_{i=1}^\ell t_i\deg(V_i)$, which is the homogenous degree of $F$ in each $\bv_i$.
The coefficient vector of $F$, regarded as a  point in a projective space,  is correspondingly called the \emph{Chow coordinate} of $V$.
Each algebraic cycle  is uniquely determined by its algebraic Chow form, in other words, determined by its Chow coordinate. 

In~\cite{Chow}, Chow proved that the set of all algebraic cycles in $\P ^n$ of  dimension $d$ and degree $m$  in the Chow coordinate space is a projective variety, 
now called the Chow variety of index $(d,m)$. 
In general, the set of all algebraic cycles in $\m A^n$ of  dimension $d$ and degree $m$ is not closed   in the Chow coordinate space. 
Below, we give a simple example.

\begin{exam}
Consider the set $X$ of all algebraic cycles in $\m A^2$ of dimension $0$ and degree $1$.
Each $V\in X$ can be represented by two linear equations $a_{i0}+a_{i1}y_1+a_{i2}y_2=0\,(i=0,1)$ with $a_{01}a_{12}-a_{02}a_{11}\neq0$.
Then the Chow form of $V$ is $F(v_{00},v_{01},v_{02})=(a_{01}a_{12}-a_{02}a_{11})v_{00}-(a_{00}a_{12}-a_{02}a_{10})v_{01}+(a_{00}a_{11}-a_{01}a_{10})v_{02}$. So the Chow coordinate of $V$ is $(a_{01}a_{12}-a_{02}a_{11},-a_{00}a_{12}+a_{02}a_{10}, a_{00}a_{11}-a_{01}a_{10})$.
Thus, the Chow coordinates of cycles in $X$ is the set $\{(c_0,c_1,c_2): c_0\neq0\}=\P^2\setminus V(c_0)$, which is not a closed variety, but is a
constructible set.
\end{exam}

The following result  shows that the set of all cycles with given degree and dimension  is always a constructible set in the  Chow coordinate space.
 \begin{prop} \label{prop-chowconstructible}
The set of all algebraic cycles in $\m A^n$  of dimension $d$ and degree $m$ is a constructible set
in a higher dimensional projective space.  We call this set the {\em affine Chow variety} of index $(d,m)$ in $\mathbb{A}^n$,
denoted by $\chow_n(d,m)$, or $\chow(d,m)$  if the space  $\m A^n$ is clear from the context.
\end{prop}
\begin{proof} 
Let $M$ be the set of all monomials in $\bv_0,\ldots,\bv_d$ which are of  degree $m$ in each $\bv_i$.
That is, $M=\{\prod_{i=0}^d\prod_{j=0}^nv_{ij}^{\sigma_{ij}}| \sigma_{ij}\in\m Z_{\geq0}, \sum_{j=0}^n\sigma_{ij}=m\}$.
Let $F_0=\sum_{\phi\in M}c_\phi\phi$ where $c_\phi$ are algebraic indeterminates over $K$.
By~\cite{Chow,Hodgevol2},  there exists a projective variety $W \subseteq  \P ^{|M|-1}$ such that
$(\bar{c}_{\phi}:\phi\in M)\in W$ if and if $\bar{F}_0=\sum_{\phi\in M}\bar{c}_\phi\phi$ is the algebraic Chow form of an algebraic cycle in $\P ^n$ of dimension $d$ and degree $m$.

Let $N=\{v_{00}^m\prod\limits_{i=1}^d\prod\limits_{j=0}^nv_{ij}^{\sigma_{ij}}| \sigma_{ij}\in\m Z_{\geq0}, \sum_{j=0}^n\sigma_{ij}=m\} \subseteq   M$
and let $\{c_1,\ldots,c_{|N|}\}$ be the set of all coefficients of $F_0$ with respect to monomials contained in $N$.
Let $W_1=W\setminus  V(c_1,\ldots,c_{|N|})$, where $V(c_1,\ldots,c_{|N|}) \subseteq  \mathbb{P}^{|M|-1}$ temporarily  denotes the projective variety defined by $c_1=\cdots=c_{|N|}=0$.
We claim that there is a  one-to-one correspondence between $\chow_n(d,m)$ and $W_1$ via algebraic Chow forms.
On the one hand, for each point in $\chow_n(d,m)$ 
corresponding to an algebraic cycle $V$, the algebraic Chow form $F=\sum_{\phi\in M}\bar{c}_\phi\phi$ of $V$ has the following Poisson-type product formula:
$F=A\prod_{\tau=1}^m(v_{00}+\sum_{j=1}^nv_{0j}\xi_{\tau j})$ where $A\in k[\bv_1,\ldots,\bv_d]$ and $(\xi_{\tau1},\ldots,\xi_{\tau n})$ is a generic point of a component of $V$.
Thus, there exists at least one monomial $\phi\in N$ such that $\phi$ appears effectively in $F$.
As a consequence, $(\bar{c}_\phi)\in W_1$.
On the other hand, for each $(\bar{c}_{\phi}:\phi\in M)\in W_1$, $\bar{F}_0=\sum_{\phi\in M}\bar{c}_\phi\phi$ is the algebraic Chow form of 
an algebraic cycle $\bar{V} '=\sum_i t_i V'_i$ in $\P ^n$ of dimension $d$ and degree $m$.
Also from the Poisson-product formula, we can see that each $V'_i\not  \subseteq  \mathbb{U}_0$, where $\mathbb{U}_0$ is the particular open set of $\mathbb{P}^n(K)$ determined by $y_0\neq 0$.
Suppose $(1,a_{i1},\ldots,a_{in})\in\mathbb{P}^n$ is a generic point of $V_i'$. Let $V_i\subset\mathbb{A}^n$ be the affine variety with $(a_{i1},\ldots,a_{in})$ as a generic point. 
Thus, $\bar{F}_0$ is the algebraic Chow form of the algebraic cycle $\bar{V}=\sum_i t_i V_i\in \chow_n(d,m)$.
Hence, we have proved that $\chow_n(d,m)$ is a constructible set.
\end{proof}

Algebraic Chow forms can uniquely determine the corresponding algebraic varieties.
In particular, the defining equations of an irreducible variety can be recovered from its Chow form; see \cite[p. 51]{Hodgevol2} or \cite[Theorem 4.45]{GDIT}.
Thus, we can obtain the following lemma which will be needed in Section \ref{deltaChow-1}.

\begin{lem} \label{lem-chowtoequation} 
The set of irreducible varieties in $\mathbb A^n$ of dimension $d$ and degree $m$ is a definable family.
%Given $C\subset\chow_n(d,m)$, the set of Chow coordinates of all irreducible varieties of dimension $d$ and degree $m$ in $\mathbb A^n$,
%then $$\{V_{c}: {c\in C} \, \text{and $c$ is the Chow coordinate of $V_{c}$}\}$$ is a definable family.  
\end{lem}
\proof Let $C\subseteq\chow_n(d,m)$ be the set of Chow coordinates of all irreducible varieties of dimension $d$ and degree $m$ in $\mathbb A^n$.
Then $C$ is a constructible set. 
Indeed, given $c\in \chow_n(d,m)$, $c$ is the Chow coordinate of an irreducible variety if and only if the corresponding
polynomial $\bar{F}_0$ with coefficient vector $c$ defined as in the proof of Proposition \ref{prop-chowconstructible}  is irreducible,
which is a definable condition. 
Given $c\in C$,  the corresponding polynomial $\bar{F}_0$ is the Chow form of some irreducible variety $V_{c}$.
By the algebraic analog of \cite[Theorem 4.45]{GDIT}, 
two subsets of polynomials in $x_1,\ldots,x_n$ with coefficients linear in $c$, say $S_1$ and $S_2$, can be computed  from $\bar{F}_0$,
such that $V_{c}$ is the Zariski closure of the quasi-variety $V(S_1)\backslash V(S_2)$.
Since Zariski closure is definable in families with respect to ACF,
$(V_{c})_{c\in C}$, the set of irreducible varieties in $\mathbb A^n$ of dimension $d$ and degree $m$,
 is a definable family.
\qed

\vskip5pt
In the remainder of this section, we recall the definitions and properties of differential  Chow forms and propose the main problem we are considering in this paper. Let  $V \subseteq   \m A^n$ be an irreducible differential variety defined over $K$ 
of dimension $d$  and $$L_i=u_{i0}+u_{i1}y_1+\cdots +u_{in}y_n\, (i=0,\ldots,d)$$ be $d+1$ differentially generic inhomogeneous linear forms.
For each $i$, denote $\textbf{u}_i=(u_{i0},u_{i1},\ldots,u_{in})$. Let
\begin{equation}
\mathcal{I}_{\bu}=[\mathbb{I}(V),L_0,\ldots,L_d]_{K\{y_1,\ldots,y_n,\textbf{u}_0,\ldots,\textbf{u}_d\}}\cap K\{\textbf{u}_0,\ldots,\textbf{u}_d\}.
\end{equation}
Then by~\cite[Lemma 4.1]{GDIT},
$\mathcal{I}_{\bu}$ is a prime differential ideal in $K \{\bu_0,\ldots,\bu_d\}$ of codimension one. 

\begin{defn}  \label{def-chow}
The {\em differential Chow form} of $V$ or $\mathbb{I}(V)$ is defined as the unique (up 
to appropriate scaling) irreducible differential polynomial $F(\bu_0,\ldots,\bu_d)$
such that $\mathcal{I}_\bu=\sat(F)$ under any differential ranking.
\end{defn} 

Note that from Definition \ref{def-chow}, the differential Chow form $F$ itself is a characteristic set of $\mathcal{J}_\bu$ under any differential ranking. 
Differential Chow forms uniquely characterize their corresponding differential ideals.
The following theorem gives some basic properties of differential Chow forms.

\begin{thm}\cite{GDIT} \label{th-chowproperty}
Let $V$ be an irreducible differential variety defined over $K$ with differential dimension $d$ and order $h$. 
Suppose $F(\bu_0,\ldots,\bu_d)$ is the differential Chow form of $V$. Then $F$
has the following properties.
\begin{itemize}
\item[1)] $\ord(F)=h$. In particular, $\ord(F,u_{i0})=h$ for each $i=0,\ldots,d$.

\item[2)] $F$ is differentially homogenous of the same degree $m$ in each $\bu_i$. 
This $m$ is called the {\em differential degree} of $V.$

\item[3)] Let $g=\deg(F,u_{00}^{(h)})$. There exist elements $\xi_{\tau j} \in \mc U$ for $\tau = 1, \ldots, g$ and $j = 1, \ldots, n$
such that  
$$F=A\prod_{\tau=1}^g(u_{00}+u_{01}\xi_{\tau1}+\cdots+u_{0n}\xi_{\tau n})^{(h)}$$ where $A$ is a differential polynomial free from $u_{00}^{(h)}$.
Moreover, each $\xi_\tau=(\xi_{\tau1},\ldots,\xi_{\tau n})$ is a generic point of $V$ and  $L_1,\ldots, L_d$ all vanish at $\xi_\tau$.

\item[4)]  The algebraic variety $B_h(V) \cap V(L_1^{(h)}, \cdots,  L_d^{(h) }, 
L_0^{(h-1)})  \subseteq   \tau_h \m A^n$ is of dimension zero.  Its cardinality,
$g$, is called the \emph{leading differential degree} of $V$. Here, the $L_i^{(j)}$ are considered as polynomials in variables $y_1^{[h]},\ldots,y_n^{[h]}$.
\end{itemize}
\end{thm}

We give a simple example to illustrate these invariants of a differential variety.

\begin{example}
Let $n=1$ and $V=\mathbb{V}(y^2y'+1) \subseteq  \mathbb{A}^1$. Then the differential Chow form of $V$ is $F(\bu_0)=u_{00}^2u_{01}u_{00}'-u_{00}^3u_{01}'-u_{01}^4$.
The order of $V$ is 1, the differential degree of $V$ is 4 and the leading differential degree of $V$ is 1.
\end{example}

The following lemma is useful in the computation of  differential Chow forms from the point view of algebraic ideals,
which follows directly from the definition and properties of the differential Chow form.  

\begin{lem} \label{lem-chowexpression}
Let $V$ be an irreducible differential variety defined over $K$ with differential dimension $d$ and order $h$. 
Let $F(\bu_0,\ldots,\bu_d)$ be the differential Chow form of $V$. Then
\begin{equation} \label{eq-chowexpression}
(F)=\big(I(B_h(V)),L_0^{[h]},\ldots,L_d^{[h]}\big)\cap K[\textbf{u}_0^{[h]},\ldots,\textbf{u}_d^{[h]}].
\end{equation}
\end{lem}
\begin{proof}
By the definition of differential Chow form, $\mathcal{I}_{\bu}=\sat(F)$ under any ranking.
Since $\ord(F)=h$ by Theorem \ref{th-chowproperty}, $\mathcal{I}_{\bu}\cap K[\textbf{u}_0^{[h]},\ldots,\textbf{u}_d^{[h]}]=(F)$.
 Let $\mathcal{J}=\big(I(B_h(V)),L_0^{[h]},\ldots,L_d^{[h]}\big)\cap K[\textbf{u}_0^{[h]},\ldots,\textbf{u}_d^{[h]}]$.
Take a generic point $\xi=(\xi_1,\ldots,\xi_n)$ of $\mathbb I(V)$ such that the $u_{ij}$ are differentially independent over $K\langle \xi\rangle$.
Let $\zeta_i=-\sum_{j=1}^nu_{ij}\xi_j$.
It is easy to show that $(\zeta_0^{[h]},\ldots,\zeta_d^{[h]})$ is a generic point of both 
$\mathcal{J}$ and $\mathcal{I}_{\bu}\cap K[\textbf{u}_0^{[h]},\ldots,\textbf{u}_d^{[h]}]$, so the two ideals are equal, which implies (\ref{eq-chowexpression}).
\end{proof}

A differential variety is called \emph{order-unmixed} if all its components have the same differential dimension and order.
Let $V$ be an order-unmixed differential variety of dimension $d$ and
order $h$ and $V=\bigcup_{i=1}^\ell V_{i}$ its minimal irreducible
decomposition with $F_{i}(\bu_{0}, \bu_{1}, \ldots, \bu_{d})$ the
differential Chow form of $V_{i}$. 
Let  \begin{equation}\label{eq-sss1}
 F(\bu_{0},\ldots,\bu_{d})=\prod_{i=1}^\ell F_{i}(\bu_{0}, \bu_{1}, \ldots, \bu_{d})^{s_{i}}
 \end{equation}
  with $s_{i}$ arbitrary nonnegative integers.
In~\cite{GDIT}, a {\em
differential algebraic cycle} is defined associated to (\ref{eq-sss1}) similar to its algebraic analog, 
that is, $\VB=\sum_{i=1}^\ell s_i V_i$ is a differential algebraic cycle with $s_i$ as the multiplicity of $V_i$ and 
$F(\bu_{0},\ldots,\bu_{d})$ is called the differential Chow form of $\VB$.

Suppose each $V_i$ is of differential degree $m_i$ and leading differential degree $g_i$,
then the leading differential degree and differential degree of $\VB$ is defined to be
$\sum_{i=1}^\ell s_ig_i$ and $\sum_{i=1}^\ell s_im_i$ respectively.

\begin{defn}  \label{def-index}
A differential cycle $\VB$ in the $n$ dimensional affine space $\mathbb{A}^n$
with dimension $d$, order $h$, leading differential degree $g$, and
differential degree $m$ is said to be of index $(d,h,g,m)$ in $\m A^n$.  
\end{defn} 

\begin{defn}
Let $\VB$ be a differential cycle of index $(d,h,g,m)$ in $\m A^n$.
The {\em differential Chow coordinate} of $\VB$ is the coefficient vector of the differential Chow form of $\VB$
considered as a point in a higher dimensional projective space determined by $(d,h,g,m)$ and $n$. 
\end{defn}

\begin{defn}  Fix an index $(d,h,g,m)$ and $n$.   To each differential 
field $k$ we associate
 the set $$\Cc_{(n,d,h,g,m)}(k) := \{\VB:  \VB \text{\,is a differential cycle of index $(d,h,g,m)$ in $\m A^n$ \text{over} $k$}\}$$
 thereby defining a functor from the category of differential fields to the category of sets.   
 If this  functor is represented by some differentially constructible set,
  meaning that there is a differentially constructible set and a natural isomorphism between the functor $\Cc(n, d, h, g, m)$ and the functor given by 
 this differentially constructible set  (regarded also as a  functor from the category of differential fields to the category of sets),  
 then we call this differentially constructible set the
 \emph{differential Chow variety} of index $(d,h,g,m)$ of $\m A^n$ and denote it by $\delta$-$\chow({n,d,h,g,m})$.
In this case, we also say the differential Chow variety $\delta$-$\chow({n,d,h,g,m})$ exists. 
\end{defn}

\vskip2pt
\begin{thm}~\cite[Theorem 5.7]{GDIT} 
In the case $g=1,$ the differential Chow variety $\delta \text{-} \chow(n,d,h,1,m )$ exists. 
\end{thm}

The above theorem was proved with constructive methods. 
But that method does not apply to the general case (when $g$ is arbitrary) and the existence of differential Chow varieties was listed as an open problem in \cite{GDIT}. 
In section 5, we will give a positive answer to the very problem using model theoretical methods, namely, showing that the differential Chow varieties exist for general cases.

\begin{rem} \label{restrictedindex}
In the algebraic setting, for an arbitrary tuple $(d,m)$, $\chow_n(d,m)$ is always a nonempty constructible set.
However, it is more subtle in the differential case and  $\Cc_{(n,d,h,g,m)}$ may be empty for certain values $(n, d,h,g,m)$. 
For example, when a differential algebraic cycle is of order 1, its differential degree is at least $2$, so $\Cc_{(n,d,1,g,1)}=\emptyset$. 
\end{rem}

\section{Degree bound for prolongation sequences} \label{confine}
We are interested in the space of all differential cycles in $n$ dimensional affine space of some fixed index $(d,h,g,m)$. 
Ultimately, the point in our parameter space corresponding to  a differential cycle $\sum_{i}a_iV_i$ will be given by the point representing $\sum_ia_iB_h (V_i) $ in an appropriate algebraic Chow variety. 
In order to ensure that the space of such algebraic varieties has the structure of a definable set, 
we must establish degree bounds for the corresponding algebraic cycles. This is the topic of the present section.

\begin{prop} \label{irr-degreebound} 
Suppose $V$ is an irreducible differential variety of index $(d,h,g,m)$ in  $ \m A^n$. Then there is a natural number $D$ depending only on $(d,h,g,m)$ such that $B_h(V)  \subseteq   \tau_h \m A^{n}$ is an irreducible algebraic variety with degree satisfying  
$\deg(B_h(V))\leq D$.
\end{prop}

\begin{proof} 
The irreducibility of $B_h(V)$ follows from the fact that $B_h(V)= V\big(\m I (V)\cap K\{ x_1, \ldots , x_n \} _{\leq h}\big)$. It remains for us to show that there is $ D$ with the claimed properties. 

Suppose $F(\bu_0,\ldots,\bu_d)$ is the differential Chow form of $V$
 where $\bu_i=(u_{i0},\ldots,u_{in})$ $\,(i=0,\ldots,d)$.
Let $\bu$ be the tuple of variables $(u_{ij})_{i=0,j=1}^{d,n}$.  That is, 
we are omitting the variables of the form $u_{i0}$.  Set $K_1=K\langle\bu\rangle$.   
Let $W$ be the differential variety in $\m A^{d+1}$ defined by $\sat(F)$ 
considered as a differential ideal in $K_1\{u_{00},\ldots,u_{d0}\}$.
Then by Theorem~\ref{th-chowproperty}, $B_h(W)=V(F) \subseteq   \tau_h\m A^{d+1}$  
is an irreducible variety.

By \cite[Theorem 4.13]{GDIT}, 
the map given by $$f( \bu_0) = ( \pd{F}{u_{0i}^{(h)}} / S_F)_{i=1}^n$$ gives a 
differential birational map from $W$ to $V_{K_1}$, the base change of $V$ to $K_1$.
By quantifier elimination in $\operatorname{DCF}_0$,
the image is given by the vanishing and non 
vanishing of some collection of differential polynomials.  By the compactness 
theorem, the number, degree and 
order of these equations and inequations must be bounded uniformly depending only on the degrees, orders, and 
number of variables of $F$ and $f$. The results of~\cite{HPnfcp} (see Remark 3.2) give a uniform upper bound, $D$ for the degree of $B_h(V)$.
\end{proof} 

\begin{corr} \label{degreebound}
Suppose $\VB=\sum_{i}a_iV_i \subseteq  \m A^n\,(a_i\in\m Z_{\geq 0})$ is an order-unmixed differential variety of index $(d,h,g,m)$. 
Then there is a natural number $D$ such that $\sum_{i}a_iB_h(V_i)$ is an algebraic cycle in $\tau_h \m A^n$ 
of  dimension $d(h+1)+h$ and degree satisfying  
$\deg(B_h(V))\leq D$
\end{corr} 

It is possible to give effective versions of Proposition~\ref{irr-degreebound} and Corollary~\ref{degreebound} with a more complicated proof; the following proposition gives such detailed effective bounds. In the following section of the paper, we will use Corollary~\ref{degreebound} to restrict the space of algebraic Chow varieties which we consider. A more detailed analysis of the particular defining equations of the differential Chow variety might be undertaken by applying the more detailed effective bounds of the following Proposition (or improving upon them), but the main thrust of our results in the next section concerns the \emph{existence} of differential Chow varieties, so the following result is primarily given to indicate that the construction of differential Chow varieties can be made effective \emph{in principle}. 

\begin{prop} \label{lonelybound} 
Suppose $V$ is an irreducible differential variety of index $(d,h,g,m)$ in  $ \m A^n$. Then $B_h(V)  \subseteq   \tau_h \m A^{n}$ is an irreducible algebraic variety with degree satisfying  
\begin{equation} \label{deg-irr}
\max\{g, m/(h+1)\}\leq \deg(B_h(V))\leq [(d+1)m]^{nh+n+1}.
\end{equation}
\end{prop} 

\begin{proof} Suppose $F(\bu_0,\ldots,\bu_d)$ is the differential Chow form of $V$
 where $\bu_i=(u_{i0},\ldots,$ $u_{in})$ $\,(i=0,\ldots,d)$.
Let $\bu= (u_{ij})_{i=0,j=1}^{d,n}$ and $K_1=K\langle\bu\rangle$.
    
Let $\mathcal{J}=[\m I(V),L_0,\ldots,L_d] \subseteq  K_1\{x_1, \ldots, x_n, u_{00},\ldots,u_{d0}\}$.
Then by the proof of \cite[Theorem 4.36]{GDIT},
the polynomials $g_{jk}=\frac{\partial
F}{\partial u_{00}^{(h)}}x_{j}^{(k)}+\sum_{\ell=1}^k {h-\ell\choose
k-\ell}/{h\choose k}\frac{\partial F}{\partial
u_{00}^{(h-\ell)}}x_{j}^{(k-\ell)}-\frac{\partial F}{\partial
u_{0j}^{(h-k)}}\,(j=1,\ldots,n;k=0,\ldots,h)$ are contained in $\mathcal{J}$.
Fix an ordering of algebraic indeterminates so that $x_1 < \cdots < x_n < x_1^{(1)} < \cdots < x_n^{(1)} < \cdots < x_1^{(h)} < \ldots < x_n^{(h)}$
and $u_{ij}^{(k)} < x_\ell^{(m)}$ for all $i, j, k, \ell,$ and $m$.

Let $\mathcal{J}^{[h]}:=\mathcal{J}\cap K_1 \{x_1, \ldots, x_n, u_{00},\ldots,u_{d0} \}_{\leq h}$. Since for each $f\in \mathcal{J}^{[h]}$,
 the algebraic remainder of $f$ with respect to $g_{jk}$ is a polynomial in $\mathcal{J}\cap K_1 \{ u_{00},\ldots,u_{d0} \}_{\leq h} = (F)$,
$\{ F \} \cup  \{ g_{jk} ~:~ 1\leq j \leq n,  0\leq k\leq h \ \}$ constitutes an algebraic  characteristic set of $\mathcal{J}^{[h]}$. Thus,
 $$\mathcal{J}^{[h]} = \big(F,(g_{jk})_{1\leq j\leq n; 0\leq k\leq h}\big):\big(\frac{\partial F}{\partial u_{00}^{(h)}}\big)^{\infty}.$$

Since the variety defined by the ideal $\big(F,(g_{jk})_{1\leq j\leq n; 0\leq k\leq h}\big):\big(\frac{\partial F}{\partial u_{00}^{(h)}}\big)^{\infty}$ is a component of the closed set given by the vanishing of $\big(F,(g_{jk})_{1\leq j\leq n; 0\leq k\leq h}\big)$, by~\cite[Theorem 1]{Heintz},
 \begin{eqnarray*} 
\deg(\mathcal{J}^{[h]}) & \leq & \deg\big(\big(F,(g_{jk})_{1\leq j\leq n; 0\leq k\leq h}\big)\big) \\
&\leq & \deg(F)^{n(h+1)+1}\leq [(d+1)m]^{nh+n+1}. \end{eqnarray*}
Since $\mathcal{J}^{[h]}\cap K_1 \{ x_1, \ldots, x_n \}_{\leq h} 
=I(B_h(V)_{K_1})$,
by~\cite{Heintz,li2011sparse}, $\deg(B_h(V))\leq\deg(\mathcal{J}^{[h]})$. 
Hence, $\deg(B_h(V)) \leq [(d+1)m]^{nh+n+1}$.

Since $\dim(B_h(V))=d(h+1)+h$ and by~\cite{GDIT}, $B_h(V)$ and some $d(h+1)+h$ hyperplanes defined by $L_0^{(i)}$ for 
$0 \leq i < h$ and  $L_1^{(i)},\ldots,L_d^{(i)}$ for $0 \leq i \leq h$ 
intersect in $g$ points, $\deg(B_h(V))\geq g$.
On the other hand, $F$ can be obtained from the algebraic Chow form of $B_h(V)$ 
using the strategy of specializations in~\cite[Theorem 4.2]{li2011sparse}. 
So $\deg(F)\leq (h+1)(d+1)\deg(B_h(V))$ and $\deg(F)=m (d+1)$.
Thus, (\ref{deg-irr}) follows.
\end{proof} 

\section{On the existence of differential Chow varieties} \label{deltaChow-1}

In this section, we will show that for a fixed $n \in \m N$ and 
a fixed index $(d,h,g,m)$, $\Cc_{(n,d,h,g,m)}$ is represented by a differentially constructible set.
That is, the differential Chow variety $\delta \text{-} \chow(n,d,h,g,m)$ exists.

Consider the disjoint union of algebraic constructible sets $$\mathcal{C}=\bigcup_{e \leq D}\chow_{n(h+1)}(d(h+1)+h,e)$$ 
where $D$ is the bound of Corollary~\ref{degreebound} and  $\chow_{n(h+1)}(d(h+1)+h,e)$ is the affine algebraic Chow variety of index $(d(h+1)+h,e)$ in $\tau_h\mathbb{A}^n$ as defined in Proposition 3.4.
So each point $a\in \mathcal{C}$ represents an algebraic cycle of the form $\sum_{i}t_iW_i$, where each $W_i$ is an irreducible  variety in $\tau_h\mathbb{A}^n$. Moreover, for a fixed $i$, if $W_i$ is prolongation admissible, by Lemma \ref{uniquegenericomponent}, the differential variety corresponding to the prolongation sequence generated by $W_i$  
has a unique dominant component $V_i\subseteq\mathbb A^n$.

Let $\mathcal{C}_1$ be the subset consisting of all points $a \in\mathcal{C}$ such that 
\begin{enumerate}
\item $a$  is the Chow coordinate of an algebraic cycle $\sum_{i}t_iW_i$ where each $W_i$  is irreducible and prolongation admissible and

\item for each $i$, the unique  dominant component of the differential variety corresponding to the prolongation sequence generated by $W_i$ is of index $(d,h,g_i,m_i)$ and $\sum_it_ig_i=g, \sum_it_im_i=m$.
\end{enumerate}

\begin{thm}\label{deltachowexists}
The set $\mc C_1$ is differentially constructible and 
the map which associates a differential algebraic cycle ${\mathbf V} = \sum s_i V_i$ of 
index $(d,h,g,m)$ in $\m A^n$ with the Chow coordinate of 
the algebraic cycle $\sum s_i B_h(V_i)$ identifies $\Cc_{(n,d,h,g,m)}$ with 
$\mc C_1$.  In particular, the differential Chow variety $\delta \text{-} \chow(n,d,h,g,m)$ 
exists.
\end{thm}

\begin{proof}
%From the very definition of prolongation admissibility, it is differentially constructible
%condition.  Thus, for each $e$ the set of Chow coordinates of positive cycles of degree $e$ built from prolongation admissible varieties of dimension $d(h+1) + h$
%is a  differentially constructible subset of $\chow(d(h+1)+h,e)(\tau_h \m A^n)$.  By Lemma~\ref{ddimdef} and Lemma~\ref{dorderdef}, the set of Chow coordinates in which each irreducible variety in a given cycle corresponds generates a prolongation sequence whose corresponding differential variety has dimension $d$ and order $h$ is a differential constructible set. 
%Thus, by Proposition \ref{gandmdetermined}, 
%$\mc C_1$ is differentially constructible. 
%
%
%By Lemma~\ref{uniquegenericomponent} the differential algebraic cycle 
%${\mathbf V} = \sum s_i V_i$
%is determined by the algebraic cycle $\sum s_i B_h(V_i)$.  By Corollary~\ref{degreebound},
%this algebraic cycle belongs to $\mc C_1$. 

First, we show $\mc C_1$ is a differentially constructible set.
From the definition of Chow coordinates, we know each $\chow_{n(h+1)}(d(h+1)+h,e)$ actually represents a definable family  $S_e:=(F_c)_{c\in \chow_{n(h+1)}(d(h+1)+h,e)}$ of 
homogenous polynomials which are Chow forms of algebraic cycles in $\mathbb A^{n(h+1)}$ of dimension  $d(h+1)+h$  and degree $e$.
Recall that the Chow coordinate $c$ of a cycle is just the  coefficient vector of the Chow form $F_c$ of this cycle.
By item 5) of Fact \ref{factdef}, the family of irreducible components of the definable family $S_e$ is definable in families.
 Take an arbitrary $c\in \chow_{n(h+1)}(d(h+1)+h,e)$ and the corresponding polynomial $F_c\in S_e$ for an example. 
 Suppose $F_c$ has the irreducible decomposition $F_c=\prod_{i=1}^\ell F_{c,i}^{t_i}$, 
 then each $F_{c,i}$ is the Chow form of an irreducible variety $W_{c,i}$ and $F_c$ is the Chow form of the algebraic cycle $\sum_{i=1}^\ell t_iW_{c,i}$.
To show $\mc C_1$ is differentially constructible, we need to consider each irreducible component $F_{c,i}$ in the above.
By Lemma \ref{lem-chowtoequation}, the family of irreducible varieties $W_{c,i}$ is a definable family.
And by Lemma \ref{prolongadmissible-definable},   the family of irreducible and prolongation admissible varieties $W_{c,i}$ is a definable family.
Let $V_{c,i}$ be the unique dominant component of the differential variety corresponding to the prolongation sequence generated by $W_{c,i}$. 
Then by Lemma \ref{uniquegenericomponent} and Remark \ref{rem-uniquegenericomponent},
the differential dimension of $V_{c,i}$ is equal to $d_1=\dim(W_{c,i})-\dim\big(\pi_{h,h-1}(W_{c,i})\big)$ 
and the order of $V_{c_i}$ is equal to $\dim(W_{c,i})-d_1(h+1)$.
Since algebraic dimension is a definable property, 
the set of Chow coordinates of the  algebraic cycles, 
each of whose irreducible component is prolongation admissible and generates a
prolongation sequence such that the unique dominant component of the differential variety corresponding to this sequence is of differential dimension $d$ and order $h$,  
is a definable set.

Suppose $\delta$-$\dim(V_{c,i})=d$ and $\ord(V_{c,i})=h$. 
Let $U$ be the algebraic variety in $\mathbb{A}^n\times \big(\mathbb{P}^{(n+1)(h+1)-1}\big)^{d+1}$ defined by the defining formulae of $W_{c,i}$ and $L_0^{[h]}=0,\ldots,L_d^{[h]}=0$ with each  $L_i^{(j)}=u_{i0}^{(j)}+\sum_{k=1}^n\sum_{\ell=0}^j{j\choose \ell}u_{ik}^{(\ell)}x_k^{(j-\ell)}$ regarded as a polynomial in  variables $x_{k}^{(j)}$ and $u_{ik}^{(\ell)}$.
Since $B_h(V_{c,i})=W_{c,i}$, by  Lemma \ref{lem-chowexpression}, the Zariski closure of the image of $U$ under the following projection map
$$\pi:\mathbb{A}^n\times \big(\mathbb{P}^{(n+1)(h+1)-1}\big)^{d+1}\longrightarrow \big(\mathbb{P}^{(n+1)(h+1)-1}\big)^{d+1}$$
is an irreducible variety of codimension 1,
and the defining polynomial $F$ of $\overline{\pi(U)}$ is the differential Chow form of $V_{c,i}$.
By item 4) of Fact \ref{factdef}, the total degree of $F$ and $\deg(F,u_{00}^{(h)})$ are both definable in families; these quantities are the differential degree and the leading differential degree of  $V_{c,i}$, respectively.
So  the differential degree and the leading differential degree of  $V_{c,i}$ are definable in families.
Hence, $\mathcal{C}_1$ is a definable set, and also a differentially constructible set due to the fact that the theory DCF$_0$ eliminates quantifiers.

By Lemma~\ref{uniquegenericomponent} and Remark \ref{rem-uniquegenericomponent}, 
each algebraic cycle $\sum s_i W_i$ corresponding to a point of $\mc C_1$ determines a  differential  algebraic cycle 
$  \sum s_i V_i \in \Cc(n,d,h,g,m)(\mc U)$, where $V_i$ is the unique dominant component of the differential variety corresponding to the prolongation sequence generated by $W_i$. And on the other hand, each differential  algebraic cycle 
$\sum s_i V_i\in \Cc(n,d,h,g,m)(\mc U)$
 determines  the corresponding algebraic cycle $\sum s_i B_h(V_i)$, which is an algebraic cycle corresponding to a point of  $\mc C_1$, by Corollary~\ref{degreebound}. 
So we have established a natural one-to-one correspondence between $\Cc(n,d,h,g,m)(\mc U)$ and ${\mathcal C}_1$. 
Thus, $\Cc(n,d,h,g,m)$ is represented by the differentially constructible set  $\delta \text{-} \chow(n,d,h,g,m) := {\mathcal C}_1$.
\end{proof} 
  
%  
% \begin{rem}
% Our identification of $\delta \text{-} \chow(n,d,h,g,m)$ with ${\mathcal C}_1$ 
% gives a universal family of differential cycles of index $(d,h,g,m)$ over 
% $\delta \text{-} \chow(n,d,h,g,m)$ in the following sense.   Our construction
% gives a recipe to associate to a point of $\delta \text{-} \chow(n,d,h,g,m)$
% a differential cycle of index $(d,h,g,m)$ in $\mathbb A^n$.  For any differentially constructible 
% family of \emph{order unmixed} cycles of index $(d,h,g,m)$,  $\{ Z_b \}_{b \in B}$
% we may produce a differentially constructible  map $f:B \to \delta \text{-} \chow(n,d,h,g,m)$ 
% so that for $b \in B$, $Z_b$ is the cycle encoded by $f(b)$.  Here, ``order unmixed'' means 
% that all irreducible components of the varieties making up the cycle have the same order.   
% However, it seems that to definably determine (in general) whether a family of differential cycles 
% consists of only order unmixed varieties  would require solving the Ritt problem.   
% 
% \end{rem}
%  
%  
%  

\begin{rem}
For the special case $d=n-1$, the existence of  the differential Chow variety of index $(n-1,h,g,m)$ can be  easily shown from the point of view of differential characteristic sets.
Indeed, note that each order-unmixed radical differential ideal $\mathcal{I}$ of dimension $n-1$ and order $h$ has the prime decomposition  $\mathcal{I}=\bigcap_{i=1}^t\sat(f_i)=\sat(\prod_{i=1}^tf_i)$, 
where $f_i\in K\{x_1,\ldots,x_n\}$ is irreducible and of order $h$.
 Thus, there is a one-to-one correspondence between $\Cc_{(n,n-1,h,g,m)}(K)$  and the set of all differential polynomials
$f\in K\{x_1,\ldots,x_n\}$ such that each irreducible component of $f$ is of order $h$, 
$\deg(f,\{x_1^{(h)}, \ldots, x_n^{(h)} \} )=g$ and the denomination of $f$ is equal to $m$.
Here, the{ \em denomination} of $f$ is the smallest number $r$ such that $x_{0}^rp(x_{1}/x_{0},\ldots,x_{n}/x_{0})\in \mathcal{F}\{x_{0},x_{1},\ldots,x_{n}\}$\cite{kolchin1992problem}.
Since all these characteristic numbers are definable for differential polynomials, $\Cc_{(n,n-1,h,g,m,n)}$ is a definable subset of $\m A^{{m+n(h+1)\choose n(h+1)}}$.
Hence, the differential Chow variety of index $(n-1,h,g,m)$ exists.
\end{rem}

\begin{rem}
We remark here that if the Kolchin closure is proved to be definable in families, then the proof of 
the existence of differential Chow varieties could be  greatly simplified and instead of using the 
algebraic Chow varieties to parameterize differential Chow varieties, one could directly show 
that the differential Chow coordinates of differential cycles of certain index constitute 
a differentially constructible set.

\end{rem}

\bibliography{Research}{}
\bibliographystyle{plain}

\setcounter{section}{0}
\setcounter{thm}{0}
\renewcommand{\theequation}{\thesection.\arabic{equation}}
\setcounter{figure}{0}
\setcounter{table}{0}

\appendix

\section*{Appendix: Geometric irreducibility and Zariski closure are definable in families \\
by William Johnson} 
\renewcommand{\thesection}{A}

In this appendix we establish the results on definability in algebraically closed fields stated as Fact~\ref{factdef} in the main text.  
We assume that the readers are familiar with model theoretic notion and we follow standard model theoretic notations and conventions.  
For example, we write $RM(a/B)$ for the Morley rank of the 
type of $a$ over $B$ and use the nonforking symbol freely.

\subsection{Irreducibility in Projective Space}
Let $\mathbb{U}$ be a monster model of ACF.  
For ${x} \in \Pp^n(\mathbb{U})$, let $\Pp_{{x}}$ be the $(n-1)$-dimensional projective space of lines through ${x}$, and let $\pi_{{x}} : \Pp^n \setminus \{{x}\} \to \Pp_{{x}}$ be the projection.

\begin{lemma}\label{first-lemma}
Let $A$ be a small set of parameters, and suppose ${x} \in \Pp^n(\mathbb{U})$ is generic over $A$.  Suppose $V$ is an $A$-definable Zariski closed subset of $\Pp^n$, of codimension \emph{greater than} 1.  Then $\pi_{{x}}(V)  \subseteq   \Pp_{{x}}$ is well-defined, Zariski closed, of codimension one less than the codimension of $V$.  Moreover, $\pi_{{x}}(V)$ is irreducible if and only if $V$ is irreducible.
\end{lemma}
\begin{proof}
Replacing $A$ with $\acl(A)$, we may assume $A$ is algebraically closed, implying that the irreducible components of $V$ are also $A$-definable.

Since ${x}$ is generic, and $V$ has codimension \emph{at least} 1, ${x} \notin V$ so $\pi_{{x}}(V)$ is well-defined.  It is Zariski closed because $\Pp^n$ is a complete variety, so $V$ is complete and the image of $V$ under any morphism of varieties is closed.

\begin{claim}
Let $C$ be any irreducible component of $V$, and let ${c} \in V$ realize the generic type of $C$, over $A {x}$.  Then ${c}$ is the sole preimage in $V$ of $\pi_{{x}}({c})$.
\end{claim}
\begin{proof}
The generic type of $C$ is $A$-definable, so ${c} \forkindep_A {x}$, and therefore $RM({x}/A {c}) = RM({x}/A) = n$.  Suppose for the sake of contradiction that there was a second point ${d} \in V$, ${d} \ne {c}$, satisfying
\[ \pi_{{x}}({d}) = \pi_{{x}}({c}).\]
This means exactly that the three points ${c}$, ${d}$, and ${x}$ are colinear.  Then ${x}$ is on the 1-dimensional line determined by ${c}$ and ${d}$, so
\[ RM({x}/A {c} {d}) \le 1.\]
But then
\[ n = RM({x}/A {c}) \le RM({x}{d}/A {c}) = RM({x}/A {c} {d}) + RM({d}/A {c}) \le 1 + RM(V) < n,\]
by the codimension assumption.
\end{proof}

Using the claim, we see that $\pi_{{x}}(V)$ and $V$ have the same dimension (= Morley rank).  Indeed, let ${v} \in V$ have Morley rank $RM(V)$ over $A {x}$.  Then ${v}$ realizes the generic type of \emph{some} irreducible component $C$, so by the claim, ${v}$ is interdefinable over $A {x}$ with $\pi_{{x}}({v})$.  But then
\[ RM(\pi_{{x}}(V)) \ge RM(\pi_{{x}}({v})/A {x}) = RM({v}/A {x}) = RM(V),\]
and the reverse inequality is obvious.  So the codimension of $\pi_{{x}}(V)$ is indeed one less.

Let $C_1, \ldots, C_m$ enumerate the irreducible components of $V$ (possibly $m = 1$).  Each of the components $C_i$ is a closed subset of $\m P^n$, and so by completeness each of the images $\pi_{{x}}(C_i)$ is a Zariski closed subset of $\Pp_{{x}}$. The image of each of the components is irreducible, on general grounds.  If $\pi_{{x}}(C_i)  \subseteq   \pi_{{x}}(C_j)$ for some $i \ne j$, then the generic type of $C_i$ would have the same image under $\pi_{{x}}$ as some point in $C_j$, contradicting the Claim.  So $\pi_{{x}}(C_i) \not \subseteq   \pi_{{x}}(C_j)$ for $i \ne j$.  It follows that the images $\pi_{{x}}(C_i)$ are the irreducible components of
\[ \pi_{{x}}(V) = \bigcup_{i = 1}^m \pi_{{x}}(C_i).\]
Therefore, $\pi_{{x}}(V)$ and $V$ have the same number of irreducible components, proving the last point of the lemma.
\end{proof}

\begin{theorem}\label{first-theorem}
Let $X_{{a}}  \subseteq   \mathbb{P}^n$ be a definable family of Zariski closed subsets of $\mathbb{P}^n$.  Then the set of ${a}$ for which $X_{{a}}$ is irreducible, is definable.
\end{theorem}
\begin{proof}
Dimension is definable in families, because ACF is strongly minimal.  So we may assume that all (non-empty) $X_{{a}}$ have the same (co)dimension.  We proceed by induction on codimension, allowing $n$ to vary.

For the base case of codimension one, we note the following: 
\begin{enumerate}
\item \label{en:a} The family of Zariski closed subsets of $\mathbb{P}^n$ is ind-definable, that is a small (i.e. less than the size of the monster model) union of definable families, because the Zariski closed subsets are exactly the zero sets of finitely-generated ideals.
\item \label{en:b} Using \ref{en:a}, the family of \emph{reducible} Zariski closed subsets of $\mathbb{P}^n$ is also ind-definable, because a definable set is a \emph{reducible} Zariski closed set if and only if it is the union of two incomparable (with respect to containment) Zariski closed sets.
\item \label{en:c} Whether or not a polynomial in $\mathbb{C}[y_1,\ldots,y_{n+1}]$ is irreducible, is definable in terms of the coefficients, because we only need to quantify over lower-degree polynomials.
\item \label{en:d} A hypersurface in $\mathbb{P}^n$ is irreducible if and only if it is equal to the zero-set of an irreducible homogeneous polynomial.  It follows by \ref{en:c} that the family of irreducible codimension 1 closed subsets of $\mathbb{P}^n$ is ind-definable.
\item By \ref{en:b} (resp. \ref{en:d}), the set of ${a}$ such that $X_{{a}}$ is reducible (resp. irreducible) is ind-definable.  Since these two sets are complementary, both are definable, proving the base case.
\end{enumerate}

For the inductive step, suppose that irreducibility is definable in families of codimension one less than $X_{{a}}$.  By choosing an isomorphism between $\Pp_{{x}}$ and $\Pp^{n-1}$, one easily verifies the definability of the set of $({x},{a})$ such that $\pi_{{x}}(X_{{a}})$ is irreducible and has codimension one less.

By Lemma~\ref{first-lemma}, $X_{{a}}$ is irreducible if and only if $({x},{a})$ lies in this set, for generic ${x}$.  Definability of types in stable theories then implies definability of the set of ${a}$ such that $X_{{a}}$ is irreducible.
\end{proof}

\begin{corollary}\label{first-corollary}
The family of irreducible closed subsets of $\mathbb{P}^n$ is ind-definable.
\end{corollary}
\begin{proof}
The family of closed subsets is ind-definable, and by Theorem~\ref{first-theorem} we can select the irreducible ones within any definable family.
\end{proof}

\begin{corollary}\label{second-corollary}
The family of pairs $(X,\overline{X})$ with $X$ definable and $\overline{X}$ its Zariski-closure, is ind-definable.
\end{corollary}
\begin{proof}
By quantifier elimination in ACF, any definable set $X$ can be written as a union of sets of the form $C \cap U$ with $C$ closed and $U$ open.  Replacing $V$ with a union of irreducible components, and distributing, we can write $X$ as a union $\bigcup_{i = 1}^m C_i \cap U_i$, with $C_i$ Zariski closed and $U_i$ Zariski open.  We may assume that $C_i \cap U_i \ne \emptyset$ for each $i$, or equivalently, that $C_i \setminus U_i \ne C_i$.

In any topological space, closure commutes with unions, so
\[ \overline{X} = \bigcup_{i = 1}^n \overline{C_i \cap U_i}.\]
Now $\overline{C_i \cap U_i}  \subseteq   \overline{C_i} = C_i$, and
\[ C_i = \overline{C_i \cap U_i} \cup (C_i \setminus U_i),\]
so by irreducibility of $C_i$, $\overline{C_i \cap U_i} = C_i$.  Therefore,
\[ \overline{X} = \bigcup_{i = 1}^n C_i.\]

Corollary~\ref{first-corollary} implies the ind-definability of the family of pairs
\[ \left(\bigcup_{i = 1}^n \overline{C_i \cap U_i}, \bigcup_{i = 1}^n C_i \right) \]
with $C_i$ irreducible closed, $U_i$ open, and $C_i \cap U_i \ne \emptyset$.  We have seen that this is the desired family of pairs.
\end{proof}

The following corollary is an easy consequence:
\begin{corollary}\label{main-point}
Let $X_{{a}}$ be a definable family of subsets of $\mathbb{P}^n$.  Then the Zariski closures $\overline{X_{{a}}}$ are also a definable family.
\end{corollary}
\subsection{Irreducibility in Affine Space}
\begin{theorem} \label{mainapp}
Let $X_{{a}}$ be a definable family of subsets of affine $n$-space.
\begin{enumerate}
\item The family of Zariski closures $\overline{X_{{a}}}$ is also definable.
\item The set of ${a}$ such that $\overline{X_{{a}}}$ is irreducible is definable.  More generally, the number of irreducible components of $\overline{X_{{a}}}$ is definable in families (and bounded in families).
\item Dimension and Morley degree of $X_{{a}}$ are definable in ${a}$.
\item If each $\overline{X_{{a}}}$ is a hypersurface given by the irreducible polynomial $F_{{a}}(y_1,\ldots,y_n)$, then the degree of $F_{{a}}$ in each $y_i$ is definable in ${a}$.  In fact, the polynomials $F_{{a}}$ have bounded total degree and the family of $F_{{a}}$ (up to scalar multiples) is definable.
\item The family of irreducible components of the 
Zariski closure is definable in families. 
\end{enumerate}
\end{theorem}

\begin{proof}
\begin{enumerate}
\item Embed $\mathbb{A}^n$ into $\mathbb{P}^n$.  Then the Zariski closure of $X_{{a}}$ within $\Aa^n$ is the intersection of $\Aa^n$ with the closure within $\Pp^n$.  Use Corollary~\ref{main-point}.
\item The number of irreducible components of the Zariski closure is the same whether we take the closure in $\Aa^n$ or $\Pp^n$.  This proves the first sentence.  The first sentence yields the ind-definability of the family of irreducible Zariski closed subsets of $\Aa^n$, from which the second statement is an exercise in compactness.
\item We may assume $X_{{a}}$ is closed, since taking the closure changes neither Morley rank nor Morley degree.  The family of $d$-dimensional Zariski irreducible closed subsets of $\Aa^n$ is ind-definable, making this an exercise in compactness.
\item Whether or not an $n$-variable polynomial is irreducible is definable in the coefficients, because to check reducibility one only needs to quantify over the (definable) set of lower-degree polynomials.  This makes the family of irreducible polynomials ind-definable.  Therefore, the set of pairs $({a},F_{{a}})$ where $F_{{a}}$ cuts out $\overline{X_{{a}}}$, is ind-definable.  For any given ${a}$, all the possibilities for $F_{{a}}$ are essentially the same, differing only by scalar multiples.  So the total degree of $F_{{a}}$ only depends on ${a}$, and compactness yields a bound on the total degree.  This in turn makes the set of pairs $({a},F_{{a}})$ definable.

\item By \cite[Proposition 3]{Heintz}, every irreducible subvariety of $\m A^n$ which is of codimension $d$ is given (set-theoretically) by the intersection the zero sets of $n+1$ polynomials, whose degrees are bounded by the degree of the variety.\footnote{By results of Storch \cite{storch1972bemerkung} or Eisenbud and Evans \cite{eisenbud1973every} and a short argument, one can improve this to $d+1$ polynomials. Degree bounds are not given in \cite{storch1972bemerkung} or \cite{eisenbud1973every}. For our purposes the existence of such bounds is what matters, so we will not pursue the details further, but merely note that the bounds of \cite{Heintz} are not tight in this case.} 
Since the degree of a family of varieties is uniformly bounded by the product, $D$, of the degrees of the defining polynomials and the number of components is bounded by the degree, there are at most $D$ many maximal irreducible subvarieties, each of which has degree less than or equal to $D$. Among such zero sets, the components of the variety are those which are maximal and irreducible. 
\end{enumerate}
\end{proof}

\end{document}